\documentclass[11pt]{article}

\usepackage[utf8]{inputenc}

\usepackage{style_scattering}

\title{Almost sure scattering for the energy critical nonlinear wave equation}
\author{Bjoern Bringmann}
\begin{document}

\allowdisplaybreaks[1]

\renewcommand{\thepage}{\arabic{page}}
\maketitle
\let\thefootnote\relax
\let\thefootnote\relax

\begin{abstract}
\noindent
We study the defocusing energy-critical nonlinear wave equation in four dimensions. Our main result proves the stability of the scattering mechanism under random pertubations of the initial data. The random pertubation is defined through a microlocal randomization, which is based on a unit-scale decomposition in physical and frequency space.  In contrast to the previous literature, we do not require the spherical symmetry of the pertubation. \\ 
The main novelty lies in a wave packet decomposition of the random linear evolution. Through this decomposition, we can adaptively estimate the interaction between the rough and regular components of the solution. Our argument relies on techniques from restriction theory, such as Bourgain's bush argument and Wolff's induction on scales. 
\end{abstract}
\tableofcontents

\section{Introduction}
We consider the defocusing cubic nonlinear wave equation in four space dimensions, that is, 
\begin{equation}\label{intro:eq_nlw}
\begin{cases}
-\partial_{tt} u + \Delta u = u^3 \qquad (t,x)\in \reals\times \reals^4~,\\
u|_{t=0}  = u_0 \in \dot{H}^s(\rfour), \quad \partial_t u|_{t=0} = u_1 \in \dot{H}^{s-1}(\rfour)~. 
\end{cases}
\end{equation}
If \( u \) is a regular solution of \eqref{intro:eq_nlw}, then it conserves the energy
\begin{equation}\label{intro:energy}
E[u](t):= \int_{\rfour} \frac{|\nabla u(t,x)|^2}{2} + \frac{(\partial_t u(t,x))^2}{2} + \frac{u(t,x)^4}{4} \dx ~. 
\end{equation}
From the Sobolev embedding \( \dot{H}^1(\rfour) \hookrightarrow L^4(\rfour) \), it follows that the initial data has finite energy if and only if \( (u_0,u_1)\in \dot{H}^1(\rfour)\times L^2(\rfour) \). Thus, we also refer to \( \dot{H}^1(\rfour)\times L^2(\rfour)\) as the energy space. In addition to the energy conservation law, \eqref{intro:eq_nlw} obeys the scaling symmetry \( u(t,x) \mapsto u_\lambda(t,x) = \lambda u(\lambda t, \lambda x) \). Since the scaling leaves the energy invariant, the equation is called energy critical. Due to the positive sign in front of the potential term \( u^4 \), we call \eqref{intro:eq_nlw} defocusing. There also exists analogues of \eqref{intro:eq_nlw} with a power-type nonlinearity in any dimension \( d \geq 3 \). \\

The Cauchy problem for (deterministic) initial data in the energy space is well-understood.  We summarize the relevant results in the following theorem. 
\begin{thm}[{Global well-posedness and scattering \cite{BG99,Grillakis90,Grillakis92,Rauch81,SS93,SS94,Strauss68,Struwe88,Tao06b}}]\label{intro:thm_deterministic}
Let \( (u_0, u_1) \in \dot{H}^1(\rfour)\times L^2(\rfour) \). Then, there exists a maximal time interval of existence \( I \) and a unique solution \( u\colon I\times \rfour \rightarrow \reals \) of \eqref{intro:eq_nlw} such that \( u \in C_t^0 \dot{H}_x^1(I \times \rfour) ~ \mathsmaller{\bigcap} ~ L_ {t,\operatorname{loc}}^3 L_x^6(I\times \rfour) \) and \( \partial_t u \in C_t^0 L_x^2(I\times \rfour) \). Furthermore, we have that 
\begin{enumerate}
\item \( u \) is global, i.e., \( I = \reals \). \label{intro:item_deterministic_gwp}
\item \( u \) obeys a global space-time bound of the form
\begin{equation*}
\| u \|_{L_t^3 L_x^6(\reals\times \rfour)} \leq C(E[u_0,u_1])~. 
\end{equation*}
\item \( u \) scatters to a solution of the linear wave equation. Thus, there exist scattering states \( (u_0^\pm, u_1^\pm ) \in \dot{H}^1(\rfour)\times L^2(\rfour) \) s.t.
\begin{equation*}
\lim_{t\rightarrow \pm \infty} \| (u(t)-W(t)(u_0^\pm,u_1^\pm), \partial_t u(t) - \partial_t W(t) (u_0^\pm,u_1^\pm)) \|_{\dot{H}^1\times L^2} = 0~. 
\end{equation*}
Here, \(W(t)(u_0^\pm,u_1^\pm) = \cos(t|\nabla|) u_0^\pm + (\sin(t|\nabla|)/|\nabla|)u_1^\pm \) denotes the solution to the linear wave equation with initial data \( (u_0^\pm,u_1^\pm) \). 
\end{enumerate}
\end{thm}

Global well-posedness and scattering results such as Theorem \ref{intro:thm_deterministic} are known for many defocusing dispersive partial differential equations, and hold for the energy-critical nonlinear Schrödinger equation \cite{Bourgain99,CKSTT08,VR07,Visan07}, the mass-critical nonlinear Schrödinger equation \cite{Dodson12,Dodson16b,Dodson16,KTV08,KVZ08}, the mass-critical generalized KdV \cite{Dodson17}, and the \( \dot{H}^{\frac{1}{2}} \)-critical radial nonlinear wave equation \cite{Dodson18}. \\
Since Theorem \ref{intro:thm_deterministic} provides a complete description of the Cauchy problem with initial data in the energy space, we now seek a similar result for initial data in a rough Sobolev space \( H_x^s \times H_x^{s-1} \), where \( s \in [0,1) \). However, since this leads to a scaling super-critical problem,  all of the above properties can fail. In fact, \cite{CCT03} proved that \eqref{intro:eq_nlw} exhibits norm inflation, which means that arbitrarily small data in \( H^s \times H^{s-1} \) can grow arbitrarily fast. More precisely, we have for all \( \epsilon > 0 \) that there exists Schwartz initial data \( (u_0,u_1) \) and a time \( 0 < t_\epsilon < \epsilon \)  such that 
\( \| (u_0,u_1)\|_{H^s\times H^{s-1}} < \epsilon \) and \( \| (u(t_\epsilon),\partial_t u(t_\epsilon)) \|_{H^s \times H^{s-1}} > \epsilon^{-1} \).  Using finite speed of propagation, one may then also construct solutions whose \( H^s \times H^{s-1} \)-norm blows up instantaneously.

\subsection{The random data Cauchy problem}

Many researchers in dispersive partial differential equations have recently examined whether blow-up behaviour, such as the norm-inflation described above, occurs for generic or only exceptional sets of rough initial data. To quantify this, one is quickly lead to random initial data. Indeed, one natural form of rough initial data is \( (u_0+ f_0^\omega, u_1+ f_1^\omega) \), where the functions \( (u_0,u_1) \in \dot{H}^1\times L^2 \) are regular and deterministic, while the functions \( (f_0^\omega,f_1^\omega)\in H^s\times H^{s-1} \) are rough and random. An analogue of Theorem \ref{intro:thm_deterministic} in this case would imply the stability of the scattering mechanism under a perturbation by noise. 

The literature on random dispersive partial differential equations is vast. We refer the interested reader to the survey \cite{BOP18}, and mention the related works \cite{BOP15, BOP17, Bourgain94,Bourgain96,BB14,Bringmann18b,BT08I,BT08II,CCMNS18,LM13,LM16,NORS12,NPS13,Pocovnicu17}. 
In the following discussion, we focus on the Wiener randomization \cite{BOP2014,LM13} of a function \( f \in H^s(\rd) \). Let \( \varphi \in C^\infty(\rd) \) be a smooth and symmetric function satisfying \( \varphi|_{[-3/8,3/8]^d}=1 \), \( \varphi|_{\rd\backslash [-5/8,5/8]^d} = 0 \), and \( \sum_{k\in \zd} \varphi(\xi-k) = 1 \) for all \( \xi \in \rd \). We then define the associated operator \( P_k \) by 
\begin{equation*}
\widehat{P_k f}(\xi) := \varphi(\xi-k) \widehat{f}(\xi)~. 
\end{equation*}
Since the translates \( \{ \varphi(\cdot-k)\}_k \) form a partition of unity, we have that 
\begin{equation}\label{intro:eq_wiener_decomposition}
f= \sum_{k\in \zd} P_k f ~, 
\end{equation}
which is called the Wiener decomposition of \( f \). The Wiener randomization is obtained by randomizing the coefficients in \eqref{intro:eq_wiener_decomposition}. Let \( I \subseteq \zd \) by an index set such that \( \zd = I~ \raisebox{0.25ex}{$\mathsmaller{\dot{\bigcup}}$}~ \{ 0 \} ~ \raisebox{0.25ex}{$\mathsmaller{\dot{\bigcup}}$} ~ (-I) \). Let \( \{ X_k \}_{k\in I \cup \{ 0 \}} \) be a sequence of symmetric, independent, and uniformly sub-gaussian random variables (see Definition \ref{prelim:def_sub_gaussian}). We set \( X_{-k} := \overline{X_k} \) for all \( k \in I \), and assume that \( X_0 \) is real-valued. Then, the Wiener randomization \( f^W\) is defined as
\begin{equation}\label{intro:eq_wiener_randomization}
f^W := \sum_{k\in \zd} X_k \cdot P_k f~. 
\end{equation}
The reason for introducing the set \( I \) is to preserve the real-valuedness of \( f \).  The Wiener randomization \( f^W\) is a random linear combination of functions with unit-scale frequency uncertainty, and therefore resembles a random Fourier series. We then examine the random data Cauchy problem

\begin{equation}\label{intro:eq_nlw_random}
\begin{cases}
-\partial_{tt} u + \Delta u = u^3 \qquad (t,x)\in \reals\times \reals^4~\\
u|_{t=0}  = u_0 + f_0^W, \quad \partial_t u|_{t=0} = u_1 +f_1^W
\end{cases}~. 
\end{equation}

We now seek an almost sure version of Theorem \ref{intro:thm_deterministic} for \eqref{intro:eq_nlw_random}. Before we summarize the recent results, let us sketch the overall strategy, which was developed by Pocovnicu in \cite{Pocovnicu17}.
 We let \( F:= \cos(t|\nabla|) f_0^W + (\sin(t|\nabla|)/|\nabla|) f_1^W \) be the solution of the linear wave equation with the rough and random initial data. We then define the nonlinear component \( v \) by  \( v := u - F\), and obtain the forced nonlinear wave equation
\begin{equation}\label{intro:eq_nlw_forced}
\begin{cases}
-\partial_{tt} v + \Delta v = (v+F)^3 \qquad (t,x)\in \reals\times \reals^4~,\\
v|_{t=0}  = u_0 , \quad \partial_t v|_{t=0} = u_1 . 
\end{cases}
\end{equation}
At the cost of introducing a rough forcing term, we have therefore removed the rough part of the initial data. This transformation is related to the Da Prato-Debussche trick \cite{PD02}. Due to the smoothing effect of the Duhamel integral, we hope to control the nonlinear component \( v \) in the energy space. The local well-posedness of \eqref{intro:eq_nlw_forced} follows readily from probabilistic Strichartz estimates (cf. \cite{BOP2014,LM13}) and a contraction mapping argument. Thus our main interest lies in the global well-posedness and the long-time behaviour of the solution. Using the deterministic well-posedness theorem and stability theory, it can be shown (cf. \cite{DLM17,Pocovnicu17}) that the solution to \eqref{intro:eq_nlw_forced} exists as long as the energy of \( v \) remains bounded. Of course, due to the forcing term in \eqref{intro:eq_nlw_forced}, the energy is no longer conserved. In addition, a global bound on the energy of \( v \) implies a global bound on the \( L_t^3 L_x^6\)-norm, and hence also implies scattering. A short calculation shows that 
\begin{equation}\label{intro:eq_energy_increment}
\frac{\mathrm{d}}{\dt} E[v](t) = \int_{\rfour} (v^3 - (v+F)^3) \partial_t v ~ \dx \approx -  3 \int_\rfour F v^2 \partial_t v\dx~. 
\end{equation} 
In the formula above, we have neglected terms that contain more than a single factor of \( F \), since they are simpler to estimate. Therefore, the remaining obstacle lies in the control of the right-hand side of \eqref{intro:eq_energy_increment}. With this overall strategy in mind, we summarize the recent literature. \\

In \cite{Pocovnicu17}, Pocovnicu proved the almost sure global existence of solutions for all \( s>0 \). Using a Gronwall-type argument and a probabilistic Strichartz estimate, \eqref{intro:eq_energy_increment} leads (at top order) to the growth estimate
\begin{equation}\label{intro:eq_pocovnicu_bound}
E[v](T) \lesssim E[v](0) \exp(C \| F \|_{L_t^1 L_x^\infty([0,T]\times \rfour)}) \lesssim E[v](0) \exp(C_\omega T^{\frac{1}{2}})~. 
\end{equation}
Since this prevents the finite time blow-up of the energy, this yields an analogue of Theorem \ref{intro:thm_deterministic}.\ref{intro:item_deterministic_gwp}. Similar theorems are also known in dimension five \cite{Pocovnicu17}, dimension three \cite{OP16}, and for the high-dimensional energy-critical nonlinear Schrödinger equation \cite{BOP17b}. \\
The bound \eqref{intro:eq_pocovnicu_bound}, however, is not sufficient to obtain global control on the energy of \( v \), and hence does not prove almost sure scattering. Assuming the regularity condition \( s > \frac{1}{2} \) and that the (deterministic) data \( (f_0,f_1) \) is spherically symmetric, Dodson, Lührmann, and Mendelson proved almost sure scattering  in \cite{DLM17}. In their argument, the energy increment is estimated by
\begin{equation}\label{intro:eq_dlm_bound}
\Big| \int_0^T \int_\rfour F v^2 \partial_t v \dx \dt \Big | \lesssim \| |x|^{\frac{1}{2}} F \|_{L_t^2 L_x^\infty([0,T]\times \rfour)} ~ \| |x|^{-\frac{1}{4}} v\|_{L_t^4 L_x^4([0,T]\times \rfour)}^2~  \|\partial_t v\|_{L_t^\infty L_x^2(\reals\times \rfour)}~. 
\end{equation}
The first factor is controlled using Khintchine's inequality and a square-function estimate, and heavily relies on the spherical symmetry of \( f_0 \) and \( f_1 \). The main novelty lies in the control of the second factor, and involves a double bootstrap argument in the energy and a Morawetz term. Under the bootstrap hypothesis, one can then control the second factor in \eqref{intro:eq_dlm_bound} by the square-root of the energy, and this eventually leads to a global bound. \\
The method of \cite{DLM17} has since been used in several related works. In \cite{DLM18}, Dodson, Lührmann, and Mendelson used local energy decay to improve the regularity condition to \( s > 0 \). After replacing the cubes in the Wiener randomization by thin annuli, the author proved almost sure scattering for radial data in dimension three \cite{Bringmann18}. The main new ingredient is an interaction flux estimate between the linear and nonlinear components of the solution. Finally, the almost sure scattering for the radial energy-critical nonlinear Schrödinger equation in four dimensions has been obtained in \cite{DLM18,KMV17}.

\begin{figure}[t!]
\begin{center}
\begin{tikzpicture}[scale=0.8]
\draw[xshift=0cm, yshift=-0.5cm, thick, fill opacity= 0.75, fill=green] (-2.5,0) -- (2.5,0) -- (2.5,1) -- (-2.5,1) -- (-2.5,0);
\draw[xshift=0cm, yshift=1.5cm, thick, fill opacity= 0.75, fill=green] (-2.5,0) -- (2.5,0) -- (2.5,1) -- (-2.5,1) -- (-2.5,0);
\draw[xshift=0cm, yshift=-2.5cm, thick, fill opacity= 0.75, fill=green] (-2.5,0) -- (2.5,0) -- (2.5,1) -- (-2.5,1) -- (-2.5,0);
\draw[xshift=0cm, yshift=0.5cm, thick, fill opacity= 0.75, fill=red] (-2.5,0) -- (2.5,0) -- (2.5,1) -- (-2.5,1) -- (-2.5,0);
\draw[xshift=0cm, yshift=-1.5cm, thick, fill opacity= 0.75, fill=red] (-2.5,0) -- (2.5,0) -- (2.5,1) -- (-2.5,1) -- (-2.5,0);
\draw[ultra thick, ->] (-3,0) -- (3,0);
\draw[ultra thick, ->] (0,-3) -- (0,3);
\node at (0,3.5) {\large $\xi$};
\node at (3.5,0) {\large $x$};

\begin{scope}[xshift=12.5cm]
\draw[thick, fill opacity= 0.5, fill=green] (-1/2,-1/2) -- (1/2,-1/2) -- (1/2,1/2) -- (-1/2,1/2) -- (-1/2,-1/2);
\draw[xshift=2cm, yshift=0cm, thick, fill opacity= 0.75, fill=green] (-1/2,-1/2) -- (1/2,-1/2) -- (1/2,1/2) -- (-1/2,1/2) -- (-1/2,-1/2);
\draw[xshift=-2cm, yshift=0cm, thick, fill opacity= 0.75, fill=green] (-1/2,-1/2) -- (1/2,-1/2) -- (1/2,1/2) -- (-1/2,1/2) -- (-1/2,-1/2);
\draw[xshift=-2cm, yshift=2cm, thick, fill opacity= 0.75, fill=green] (-1/2,-1/2) -- (1/2,-1/2) -- (1/2,1/2) -- (-1/2,1/2) -- (-1/2,-1/2);
\draw[xshift=-2cm, yshift=-2cm, thick, fill opacity= 0.75, fill=green] (-1/2,-1/2) -- (1/2,-1/2) -- (1/2,1/2) -- (-1/2,1/2) -- (-1/2,-1/2);
\draw[xshift=0cm, yshift=2cm, thick, fill opacity= 0.75, fill=green] (-1/2,-1/2) -- (1/2,-1/2) -- (1/2,1/2) -- (-1/2,1/2) -- (-1/2,-1/2);
\draw[xshift=0cm, yshift=-2cm, thick, fill opacity= 0.75, fill=green] (-1/2,-1/2) -- (1/2,-1/2) -- (1/2,1/2) -- (-1/2,1/2) -- (-1/2,-1/2);
\draw[xshift=2cm, yshift=-2cm, thick, fill opacity= 0.75, fill=green] (-1/2,-1/2) -- (1/2,-1/2) -- (1/2,1/2) -- (-1/2,1/2) -- (-1/2,-1/2);
\draw[xshift=2cm, yshift=2cm, thick, fill opacity= 0.75, fill=green] (-1/2,-1/2) -- (1/2,-1/2) -- (1/2,1/2) -- (-1/2,1/2) -- (-1/2,-1/2);
\draw[xshift=1cm, yshift=1cm, thick, fill opacity= 0.75, fill=green] (-1/2,-1/2) -- (1/2,-1/2) -- (1/2,1/2) -- (-1/2,1/2) -- (-1/2,-1/2);
\draw[xshift=-1cm, yshift=1cm, thick, fill opacity= 0.75, fill=green] (-1/2,-1/2) -- (1/2,-1/2) -- (1/2,1/2) -- (-1/2,1/2) -- (-1/2,-1/2);
\draw[xshift=1cm, yshift=-1cm, thick, fill opacity= 0.75, fill=green] (-1/2,-1/2) -- (1/2,-1/2) -- (1/2,1/2) -- (-1/2,1/2) -- (-1/2,-1/2);
\draw[xshift=-1cm, yshift=-1cm, thick, fill opacity= 0.75, fill=green] (-1/2,-1/2) -- (1/2,-1/2) -- (1/2,1/2) -- (-1/2,1/2) -- (-1/2,-1/2);
\draw[xshift=1cm, yshift=0cm, thick, fill opacity= 0.75, fill=red] (-1/2,-1/2) -- (1/2,-1/2) -- (1/2,1/2) -- (-1/2,1/2) -- (-1/2,-1/2);
\draw[xshift=1cm, yshift=2cm, thick, fill opacity= 0.75, fill=red] (-1/2,-1/2) -- (1/2,-1/2) -- (1/2,1/2) -- (-1/2,1/2) -- (-1/2,-1/2);
\draw[xshift=1cm, yshift=-2cm, thick, fill opacity= 0.75, fill=red] (-1/2,-1/2) -- (1/2,-1/2) -- (1/2,1/2) -- (-1/2,1/2) -- (-1/2,-1/2);
\draw[xshift=0cm, yshift=1cm, thick, fill opacity= 0.75, fill=red] (-1/2,-1/2) -- (1/2,-1/2) -- (1/2,1/2) -- (-1/2,1/2) -- (-1/2,-1/2);
\draw[xshift=0cm, yshift=-1cm, thick, fill opacity= 0.75, fill=red] (-1/2,-1/2) -- (1/2,-1/2) -- (1/2,1/2) -- (-1/2,1/2) -- (-1/2,-1/2);
\draw[xshift=-1cm, yshift=2cm, thick, fill opacity= 0.75, fill=red] (-1/2,-1/2) -- (1/2,-1/2) -- (1/2,1/2) -- (-1/2,1/2) -- (-1/2,-1/2);
\draw[xshift=-1cm, yshift=0cm, thick, fill opacity= 0.75, fill=red] (-1/2,-1/2) -- (1/2,-1/2) -- (1/2,1/2) -- (-1/2,1/2) -- (-1/2,-1/2);
\draw[xshift=-1cm, yshift=-2cm, thick, fill opacity= 0.75, fill=red] (-1/2,-1/2) -- (1/2,-1/2) -- (1/2,1/2) -- (-1/2,1/2) -- (-1/2,-1/2);
\draw[xshift=2cm, yshift=1cm, thick, fill opacity= 0.75, fill=red] (-1/2,-1/2) -- (1/2,-1/2) -- (1/2,1/2) -- (-1/2,1/2) -- (-1/2,-1/2);
\draw[xshift=2cm, yshift=-1cm, thick, fill opacity= 0.75, fill=red] (-1/2,-1/2) -- (1/2,-1/2) -- (1/2,1/2) -- (-1/2,1/2) -- (-1/2,-1/2);
\draw[xshift=-2cm, yshift=1cm, thick, fill opacity= 0.75, fill=red] (-1/2,-1/2) -- (1/2,-1/2) -- (1/2,1/2) -- (-1/2,1/2) -- (-1/2,-1/2);
\draw[xshift=-2cm, yshift=-1cm, thick, fill opacity= 0.75, fill=red] (-1/2,-1/2) -- (1/2,-1/2) -- (1/2,1/2) -- (-1/2,1/2) -- (-1/2,-1/2);
\draw[ultra thick, ->] (-3,0) -- (3,0);
\draw[ultra thick, ->] (0,-3) -- (0,3);
\node at (0,3.5) {\large $\xi$};
\node at (3.5,0) {\large $x$};
\end{scope}
\end{tikzpicture}
\end{center}
\caption*{\small{In the left image, we display a partition of the phase space \( \rd \times \rd \) into horizontal strips, which forms the basis of the Wiener randomization. In the right image, we display a partition of \( \rd \times \rd\) into cubes, which forms the basis of the microlocal randomization. A similar figure has been used in the author's previous work \cite[Figure 1]{Bringmann18}}.}
\caption{Partions of phase space}
\label{figure:randomization}
\end{figure}

\subsection{Main result and ideas}

The remaining open question is concerned with almost sure scattering for non-radial data. In order to state the main result of this paper, we first need to introduce a microlocal randomization. While the Wiener randomization is based on a unit-scale decomposition in frequency space, the microlocal randomization is based on a unit-scale decomposition in phase space (see Figure \ref{figure:randomization}). 

\begin{definition}[Microlocal randomization]\label{intro:def_microlocal}
Let \( \{ X_{k,l} \}_{k\in I \cup \{ 0 \}, l\in \zd} \) be a sequence of symmetric, independent, and uniformly sub-gaussian random variables. We set \( X_{-k,l} := \overline{X_{k,l}} \) for all \( k \in I \), and assume that \( X_{0,l} \) is real-valued. For any \( f \in H^s(\rd) \), we define its microlocal randomization \( f^\omega \) by 
\begin{equation}\label{intro:eq_microloca}
f^\omega(x):= \sum_{k,l\in \mathbb{Z}^d} X_{k,l} P_k (\varphi(\cdot-l) f)(x)~. 
\end{equation}
\end{definition}
The microlocal randomization is inspired by \cite{Murphy17}, which used a randomization in physical space. 

\begin{thm}[Almost sure scattering for the microlocal randomization] \label{main_theorem}
Let \( (u_0,u_1)\in \dot{H}^1(\rfour)\times L^2(\rfour) \), and let \( (f_0,f_1)\in H^s(\rfour)\times H^{s-1}(\rfour) \), where \( s > \frac{11}{12} \). Then, there exists a random maximal time interval of existence \( I \) and a solution \( u \colon I \times \rfour \rightarrow \reals \) of \eqref{intro:eq_nlw_random} such that
\begin{equation*}
u\in W(t)(f_0^\omega,f_1^\omega)+ \big(C_t^0 \dot{H}_x^1(I\times \rfour) ~ \mathsmaller{\bigcap}~ L_{t,\operatorname{loc}}^3L_x^6(I\times \rfour ) \big) ~~ \text{and} ~~ \partial_t u \in \partial_t W(t) (f_0^\omega,f_1^\omega) + C_t^0 L^2_x(I\times \rfour)~. 
\end{equation*}
Furthermore, we have that 
\begin{enumerate}
\item \( u \) is almost surely global, i.e., \( I = \reals \). 
\item \( u \) almost surely satisfies the global space-time bound
\( \| u \|_{L_t^3 L_x^6(\reals\times \rfour)} < \infty \)~.
\item \( u \) almost surely scatters to a solution of the linear wave equation. Thus, there exist random scattering states \( (u_0^\pm, u_1^\pm ) \in \dot{H}^1(\rfour)\times L^2(\rfour) \) s.t.
\begin{equation*}
\lim_{t\rightarrow \pm \infty} \| (u(t)-W(t)(u_0^\pm+f_0^\omega,u_1^\pm+f_1^\omega), \partial_t u(t) - \partial_t W(t) (u_0^\pm+f_0^\omega,u_1^\pm+f_1^\omega)) \|_{\dot{H}^1\times L^2} = 0~. 
\end{equation*}
\end{enumerate}
\end{thm}
While Theorem \ref{main_theorem} is only proven for the microlocal randomization, the majority of our argument directly applies to the Wiener randomization. \\
The main novelty in this paper lies in the application of a wave packet decomposition. To illustrate this idea, fix some \( k \in \zd \) with \( \| k\|_\infty \sim N \), and assume that \( \widehat{f}_k(\xi) = N^{-s} \varphi(\xi-k) \). Then, \( f_k \) will essentially be unaffected by both the Wiener and microlocal randomizations, and hence forms an important example. From the method of non-stationary phase, it follows for all times \( t\in [0,N] \) that the evolution \( \exp(\pm it|\nabla|) f_k \) is concentrated in the ball \( |x\pm t k/\|k\|_2|\lesssim 1 \), and has amplitude \( \sim N^{-s} \). In space-time, we can therefore view the evolution as a tube, see Figure \ref{subfigure:single}. For larger times, the dispersion of the evolution becomes significant, and the physical localization deteriorates. The wave packet perspective also explains the effect of the frequency randomization on the evolution. In Figure \ref{subfigure:bush}, we display a bush (cf. \cite{Bourgain91}), which is a collection of wave packets intersecting at a single point. If all wave packets in the bush have comparable amplitudes and the data is deterministic, one expects that the \( L_t^\infty L_x^\infty \)-norm is proportional to the number of wave packets. For random data, however, the phases of the wave packets are all independent, and the central limit theorem predicts that the \( L_t^\infty L_x^\infty\)-norm should instead be proportional to the square-root of the number of wave packets. \\

\begin{figure}[t!]
\begin{subfigure}{0.5\textwidth}
\begin{center}
\includegraphics[height=5cm]{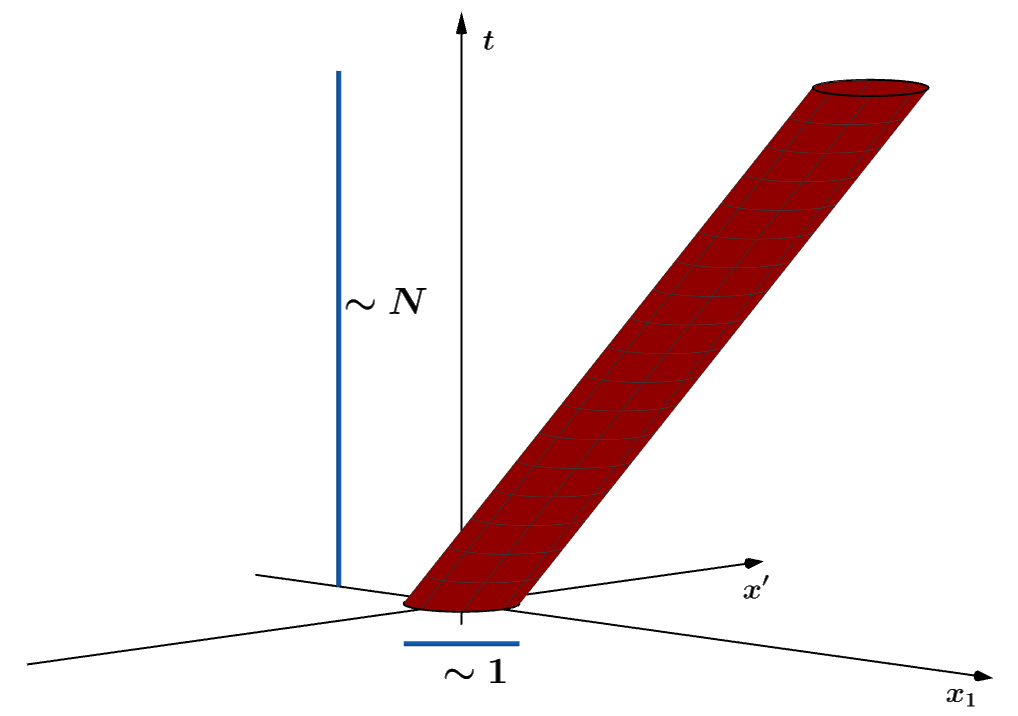}
\end{center}
\subcaption{Single wave packet}\label{subfigure:single}
\end{subfigure}
\begin{subfigure}{0.5\textwidth}
\includegraphics[height=5cm]{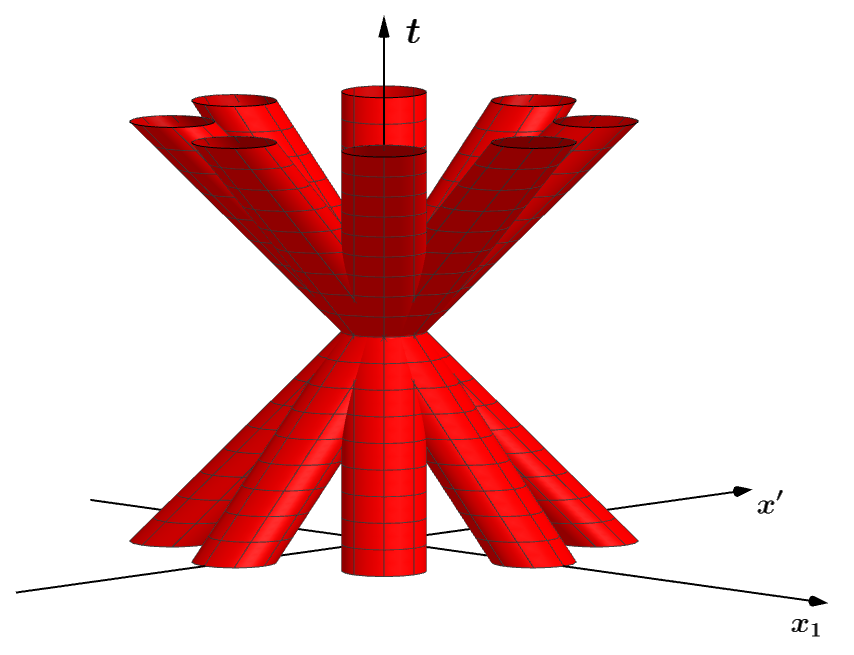}
\subcaption{Bush}
\label{subfigure:bush}
\end{subfigure}
\caption*{\small{In (a), we display the evolution  \( \exp(\pm it|\nabla|) f_k \) on the time-interval \( [0,N] \).  The space-time support can be viewed as a tube of length \( \sim N \) and width \( \sim 1\). The spatial center travels in a fixed direction at the speed of light, which has been normalized to \( 1 \).  Furthermore, the amplitude of the evolution is given by \( \sim N^{-s} \). In (b), we display a so-called bush, which is a collection of wave packets intersecting at a single point. }}
\caption{Wave packet heuristic}
\label{figure:single_wave_packet}
\end{figure}

The examples in Figure \ref{figure:single_wave_packet} also illustrates an important heuristic: The natural timescale for the randomized evolution at frequency \( N \) is \( T=N\). This differs from the natural timescale predicted by the (deterministic) bump-function heuristic, which is \( T= N^{-1} \). We therefore decompose the positive time-interval as 
\begin{align}\label{intro:eq_division_time}
[0,\infty) = \Big( \bigcup_{n=0}^{\lfloor N^{\theta} \rfloor} [nN,(n+1)N) \Big) ~\mathsmaller{\bigcup} ~ [N^{1+\theta},\infty)~, 
\end{align}
where \( \theta >0 \) is a parameter. Our argument then splits into two separate parts.\\

On the long-time interval \( [N^{1+\theta},\infty) \), we use the additional decay obtained through the physical randomization. The basic idea is that after such a long time, the linear evolution could only be concentrated through constructive interference of a large portion of the initial data, which is highly unlikely due to the physical randomness (see Figure \ref{figure:physical}). To make this rigorous, we prove an \( L_t^1 L_x^\infty([N^{1+\theta},\infty)\times \rfour)\)-bound on \( P_N F \), and this is sufficient to control the energy increment. This part of the proof requires the condition \( s > 1- \theta/2 \). \\

The majority of this paper focusses on time intervals such as \( [0,N) \). This part of the argument  does not rely on the physical randomness, and therefore also applies to the Wiener randomization. On this interval, we decompose the evolution into a family of wave packets, see Figure \ref{figure:multi_packet}. As can be seen from a single wave packet, we cannot (always) control the evolution in \( L_t^1 L_x^\infty \). Instead, we use the following dichotomy: Either \( F \) consists of only a few wave packets, in which case its support lies on a few light-cones, or it consists of many wave packets, in which case the \( L_t^\infty L_x^\infty\)-norm should be small.  \\

We now present a heuristic and simplified version of the main argument. 
In order to illustrate the ideas, let us first assume that all wave packets belong to a single frequency \( k\in \zd\). After a dyadic decomposition, we may further assume that all wave packets have amplitudes comparable to \( 2^m \). Using the same notation as in Section \ref{section:wp}, we denote the number of wave packets with this amplitude by \( \# \Am \) . Due to the \( L^2\)-orthogonality of the wave packets, we have that \( 2^m (\# \Am)^{\frac{1}{2}} \lesssim N^{-s}\). \\
In the case of only a few wave packets, we control the contribution on each tube separately. We have that 
\begin{equation*}
\Big| \int_{0}^{N} F_N v^2 \partial_t v \dx \dt \Big| \lesssim (\# \Am) N^{\frac{1}{2}} 2^m    \big( \sup_{\text{tubes}~T} \| v\|_{L_t^4 L_x^4(T)}^2 \big) \| \partial_t v \|_{L_t^\infty L_x^2([0,N)\times \rfour)} \lesssim N^{\frac{1}{2}} 2^m \# \Am  \sup_{t\in[0,N)} E[v](t)~. 
\end{equation*}
The supremum ranges over all tubes of length \( \sim N \), width \( \sim 1 \), and unit-speed direction inside  \( [0,N)\times \rfour \). Using a flux estimate and a bootstrap argument, we controll this supremum by the square-root of the energy. \\
In the case of many wave packets (with the same direction), we use that their supports are disjoint, and obtain that 
\begin{equation*}
\begin{aligned}
\Big| \int_{0}^{N} F_N v^2 \partial_t v \dx \dt \Big|& \lesssim N \| F_N \|_{L_t^\infty L_x^\infty([0,N)\times \rfour)}\| v \|_{L_t^\infty L_x^4([0,N)\times \rfour)}^2 \| \partial_t v \|_{L_t^\infty L_x^2([0,N)\times \rfour)} 
\lesssim N2^m  \sup_{t\in[0,N)} E[v](t)~. 
\end{aligned}
\end{equation*}
By combining both estimates, it follows that
\begin{align*}
\Big| \int_{0}^{N} F_N v^2 \partial_t v \dx \dt \Big| &\lesssim \min(N^{\frac{1}{2}} 2^m \# \Am, N2^m) \sup_{t\in[0,N)} E[v](t) 
\leq N^{\frac{3}{4}} 2^m (\# \Am)^{\frac{1}{2}}  \sup_{t\in[0,N)} E[v](t)~. 
\end{align*}
We insert the bound \( 2^m (\# \Am)^{\frac{1}{2}} \lesssim N^{-s} \), sum over \( N^\theta \) intervals, and arrive at the condition \( s> 3/4+\theta \). In order to match the conditions from the intervals \( [nN,(n+1)N) \) and the long-time interval \( [N^{1+\theta},\infty) \), we choose \( \theta= 1/6 \), and obtain the regularity condition \( s> 11/12\). \\

In order to remove the restriction to a single frequency, we need to consider both multiple directions and multiple scales.  For this, we rely on techniques from the literature on the Kakeya and restriction conjectures. In order to control multiple directions, we use Bourgain's bush argument \cite{Bourgain91}. The basic idea is to distinguish points which lie in multiple tubes from points which lie only in a few tubes. To this end, we group the wave packets into several bushes and a collection of (almost) non-overlapping wave packets (see Figure \ref{figure:multi_packet}). We then almost argue as for a single frequency, but also use that each bush lies on the surface of a light-cone, which is crucial for the flux estimate. In order to control multiple scales, we rely on Wolff's induction on scales strategy \cite{Wolff01}. To fix ideas, let us try to bound the energy increment \( E[v](N) -E[v](0) \). We have already described the estimates for wave packets of length greater than or equal to \( N \), but the space-time region \( [0,N] \times \rfour \) also contains many shorter wave packets. By induction on scales, we can already close the bootstrap argument at these shorter scales, which greatly reduces the complexity of the proof.  We postpone a more detailed discussion to the Sections \ref{section:wp} and \ref{section:a_priori}.

\begin{figure}[t!]
\begin{center}
\includegraphics[height=5cm]{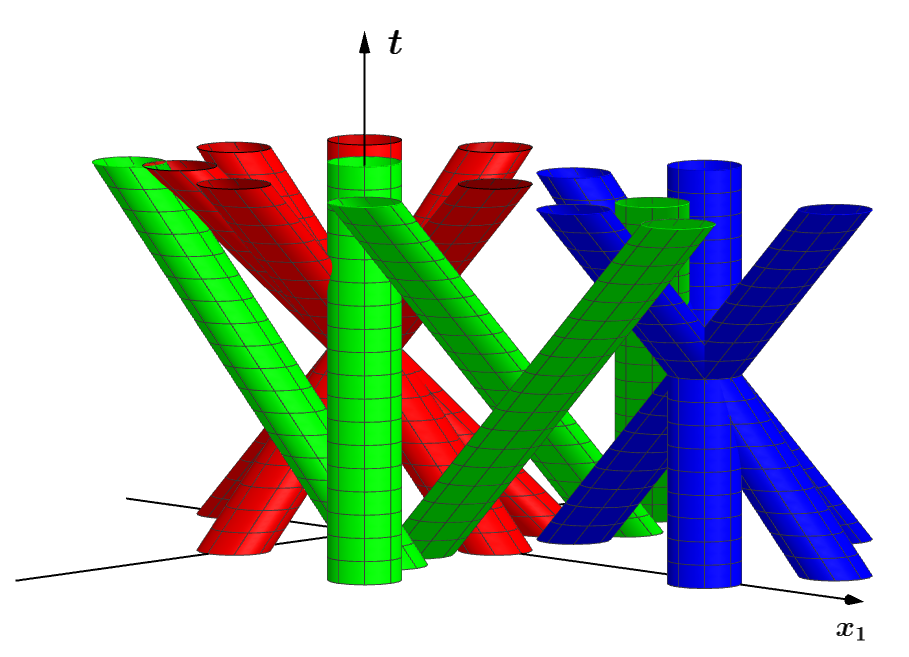}
\end{center}
\caption*{\small{We illustrate the wave packet decomposition of the linear evolution. We partition the wave packets into three groups: Two separate bushes (red and blue) and a collection of almost non-overlapping wave packets (green).}}
\caption{Wave packet decomposition}
\label{figure:multi_packet}
\end{figure}

\paragraph{Acknowledgements:} ~\\I want to thank my advisor Terence Tao for his invaluable guidance and support. In particular, he proposed the greedy selection algorithm in Section \ref{section:wp}. Furthermore, I want to thank Rowan Killip and Monica Visan for several interesting discussions. 
The figures in this paper have been created using TikZ and GeoGebra.

\section{Notation and preliminaries}\label{section:prelim}

For the rest of this paper, the positive integer \( d \) denotes the dimension of physical space. In the analysis of the nonlinear evolution, we will eventually specialize to \( d=4 \). Furthermore, we also fix positive absolute constants \( \delta, \theta, \) and \( \eta \). The parameter \( \delta \) will be used to deal with the spatial tails of the wave packets and certain summability issues. The parameter \( \theta \) is used in the division of time (see \eqref{intro:eq_division_time}). We will eventually choose \( \theta= 1/6\), but prefer to keep \( \theta \) as a free parameter until the end of the argument. Finally, \( \eta \) describes the size of the frequency truncated data, see Proposition \ref{wp:prop_smallness}. \\
If \( A, B \) are positive quantities, we write \( A \lesssim B \) if and only if there exist a constant \( C=C(\delta,\theta) \) such that \( A \leq CB \). Furthermore, most capital letters, such as \( N \), \( M \), and \( R \), will denote dyadic numbers greater than or equal to \( 1 \). \\
Finally, we define the Fourier transform \( \widehat{f} \) of an integrable function \( f \colon \rd \rightarrow \mathbb{C} \) by 
\begin{equation*}
\widehat{f}(\xi) := \frac{1}{(2\pi)^{\frac{d}{2}}} \int_{\rd} \exp(-ix\xi) f(x) \dx~. 
\end{equation*}
We now summarize a few basic results from probability theory, harmonic analysis, and dispersive partial differential equations.

\subsection{Probability theory}\label{section:probability}

We recall a few basic estimates for sub-gaussian random variables. For an accessible introduction, we refer the reader to \cite{Vershynin}. 

\begin{definition}[Sub-gaussian random variable]\label{prelim:def_sub_gaussian}
Let \( (\Omega,\mathcal{F},\mathbb{P}) \) be a probability space, and let \( X \colon (\Omega,\mathcal{F})\rightarrow \reals \) be a random variable. We then define the sub-gaussian norm by
\begin{equation}\label{prelim:eq_subgaussian}
\| X \|_{\Psi_2} := \sup_{p\geq 1} \frac{(\mathbb{E}[|X|^p])^{\frac{1}{p}}}{\sqrt{p}} 
\end{equation}
We call a random variable \( X \) sub-gaussian if and only if \( \| X \|_{\Psi_2} <\infty \). Furthermore, we call a family of random variables \( \{ X_i \}_{i\in I} \) uniformly sub-gaussian if and only if \( \sup_{i\in I} \| X_i \|_{\Psi_2} <\infty \). 
\end{definition}
The relationship to the Gaussian distribution may not be obvious from \eqref{prelim:eq_subgaussian}. However, it follows from \cite[Proposition 2.52]{Vershynin} that \eqref{prelim:eq_subgaussian} implies  
\begin{equation*}
\mathbb{P}(|X|> x) \leq 2 \exp\left( - c \frac{x^2}{\| X\|_{\Psi_2}^2} \right ) \qquad \forall  x > 0 ~. 
\end{equation*}
Many concentration inequalities for the sums of independent sub-gaussian random variables can be found in the literature. In the following, we mainly rely on Khintchine's inequality. 
\begin{lem}[{Khintchine's inequality, \cite[Corollary 5.12]{VershyninRMT} or \cite[Proposition 2.6.1 and Exercise 2.6.5]{Vershynin}}]\label{prelim:lem_khintchine}
Let  \( (X_j)_{j=1,\hdots,J} \) be a finite sequence of independent sub-gaussian random variables with zero mean. Then, it holds for all deterministic sequences \( (a_j)_{j=1,\hdots,J} \), and all \( p \geq 1 \), that 
\begin{equation}\label{prelim:eq_khintchine}
\left( \mathbb{E} \Big[  |\sum_{j=1}^J a_j X_j|^p \Big] \right)^{\frac{1}{p}} \lesssim \sqrt{p} \left( \max_{j=1,\hdots,J} \| X_j \|_{\Psi_2} \right) \left( \sum_{j=1}^J |a_j|^2 \right)^{\frac{1}{2}} 
\end{equation}
In particular, the sum \(  \sum_{j=1}^J a_j X_j \) is sub-gaussian. 
\end{lem}
In this paper, Khintchine's inequality will often be combined with Minkowski's integral inequality, which we recall below.
\begin{lem}[Minkowski's integral inequality]\label{prelim:lem_minkowski}
Let \( (X,\mu) \) and \( (Y,\nu) \) be two \( \sigma\)-finite measure spaces, and let \( 1 \leq p \leq q \leq \infty \). Then, we have for all measurable functions 
\( f \colon X\times Y \rightarrow \reals \) that 
\begin{equation*}
\| \| f(x,y) \|_{L^p(X)} \|_{L^q(Y)} \leq \| \| f(x,y) \|_{L^q(Y)} \|_{L^p(X)}~. 
\end{equation*}
\end{lem}
The special case \( p=1 \) is the standard Minkowski inequality, and it can be found in most real analysis books (see e.g. \cite[Theorem 2.4]{LL01}). Since Lemma \ref{prelim:lem_minkowski} is central to many arguments in this paper, we prove the general statement from this special case. 
\begin{proof}
Since \( q/p \geq 1 \), we have that 
\begin{equation*}
\| \| f(x,y) \|_{L^p(X)} \|_{L^q(Y)} = \| \| f(x,y)^p \|_{L^1(X)} \|_{L^{\frac{q}{p}}(Y)}^{\frac{1}{p}} \leq \| \| f(x,y)^p \|_{L^{\frac{q}{p}}(Y)} \|_{L^1(X)}^{\frac{1}{p}} = \| \| f(x,y) \|_{L^q(Y)} \|_{L^p(X)}~.
\end{equation*}
\end{proof}

We will also need  a crude bound on the maximum of dependent sub-gaussian random variables 
\begin{lem}[{Suprema of dependent sub-gaussian random variables \cite[Exercise 2.5.10]{Vershynin}}]\label{prelim:lem_suprema_subgaussian}
Assume that \( (X_j)_{j=1,\hdots,J} \) are (possibly dependent) sub-gaussian random variables. Then, 
\begin{equation}\label{prelim:eq_suprema_subgaussian}
\mathbb{E}\left[ \max_{j=1,\hdots,J} |X_j| \right] \lesssim \sqrt{\log(2+J)} \max_{j=1,\hdots,J} \| X_j \|_{\Psi_2}~. 
\end{equation}
\end{lem}
\begin{proof}
Let \(1 \leq p < \infty \). Using Hölder's inequality, we obtain that
\begin{align*}
\mathbb{E}\left[ \max_{j=1,\hdots,J} |X_j| \right]  \leq \left( \mathbb{E}\left[ \max_{j=1,\hdots,J} |X_j|^p \right] \right)^{\frac{1}{p}} 
\leq \left( \sum_{j=1}^J \mathbb{E} \left[ |X_j|^p \right] \right)^{\frac{1}{p}}
\leq J^{\frac{1}{p}} \sqrt{p} \max_{j=1,\hdots,J} \| X_j \|_{\Psi_2}~. 
\end{align*}
After choosing \( p= \log(2+J) \), we arrive at \eqref{prelim:eq_suprema_subgaussian}. 
\end{proof}

\subsection{Harmonic analysis}

Let \( N \in 2^{\mathbb{N}_0} \) and \( k \in \zd \). As in the introduction, we let \( \varphi \in C^\infty_c(\rd) \) be a smooth and symmetric function satisfying \( \varphi|_{[-\frac{3}{8},\frac{3}{8}]^d} =1\), \( \varphi|_{\rd \backslash [-\frac{5}{8},\frac{5}{8}]^d} = 0 \), and \( \sum_{k \in \zd} \varphi(\cdot - k ) = 1 \). We also define \( \psi(\xi) = \varphi(\xi) - \varphi(2\xi) \). 
Then, the re-centered Littlewood-Paley operators are defined as 

\begin{equation*}
\widehat{P_{N;k} f}(\xi) = \begin{cases} 
\begin{tabular}{ll}
 \( \psi\big(\frac{\xi-k}{N}\big) \widehat{f}(\xi) \) & if \( N > 1 \)\\
 \( \varphi(\xi-k)  \widehat{f}(\xi) \) & if \( N =1 \)
\end{tabular}
\end{cases} ~ . 
\end{equation*}
With this choice of \( \psi \), it holds that \( P_{N;k} P_{M;k} = 0 \) if \( M\geq 4N \) or \( N \geq 4M \). 
To simplify the notation, we also set  \( P_N := P_{N;0} \) and \( P_k := P_{1;k} \). Furthermore, we define the fattened Littlewood-Paley operators 
\begin{equation}\label{prelim:eq_fattened}
 \widetilde{P}_{N;k}:= \sum_{2^{-10} N\leq M \leq 2^{10} N}  P_{M;k}~.
 \end{equation}

\begin{lem}[Bernstein's inequalities]\label{prelim:lem_bernstein}
Let \( N \in 2^{\mathbb{N}_0} \), \( k \in \zd \), and \( 1 \leq p \leq q \leq \infty \). Then, we for all \( f \in L_x^p(\rd) \) that 
\begin{align}
\| P_{N;k} f \|_{L_x^q(\rd)} &\lesssim N^{d( \frac{1}{p} - \frac{1}{q} ) } \| P_{N;k} f \|_{L_x^p(\rd)}~, \label{prelim:eq_bernstein} \\
\| |\nabla| P_{N;0} f \|_{L_x^p(\rd)} &\lesssim N \| P_{N;0} f \|_{L_x^p(\rd)} ~. \label{prelim:eq_bernstein_derivative}
\end{align}
\end{lem}
We emphasize that the constant in  \eqref{prelim:eq_bernstein} is independent of \( k \in \zd \), since the phase \( \exp(ikx) \) does not affect the \( L_x^p\)-norms. \\

We also record the following standard consequence of Bernstein's inequality and the uncertainty principle.  

\begin{lem}\label{prelim:lem_inf_removal}
Let \( N \geq 1 \), let \( a,b \in \reals \), and assume that \( b\geq a+1/N \). Then, we have for all \( f \in L_x^2(\rd) \), all \( 1 \leq q < \infty\), and all \( 1 \leq p \leq \infty \), that 
\begin{equation*}
\| \exp(\pm it|\nabla|) P_N f \|_{L_t^\infty L_x^p([a,b]\times \rd)} \lesssim N^{\frac{1}{q}} \| \exp(\pm it|\nabla|) P_N f \|_{L_t^q L_x^p([a,b]\times \rd)}~. 
\end{equation*}
\end{lem}
The argument is essentially taken from \cite{Bringmann18}. 

\begin{proof}
Pick \( t_0 \in [a,b] \), and let \( I \) be any interval such that \( t_0 \in I \subseteq [a,b] \). For all \( t \in I \), it holds that 

\begin{equation*}
\exp(\pm it_0 |\nabla|) P_N f^\omega = \exp(\pm it |\nabla|) P_N f^\omega \pm i \int_{t}^{t_0} |\nabla| \exp(\pm it^\prime |\nabla|) P_N f^\omega \dtprime
\end{equation*}
From Bernstein's inequality, we obtain that
\begin{align*}
\| \exp(\pm it_0 |\nabla|) P_N f^\omega \|_{L_x^p(\rd)}  &\leq  \| \exp(\pm it |\nabla|) P_N f^\omega \|_{L_x^p(\rd)}  +  |I|^{1-\frac{1}{q}}  \| |\nabla| \exp(\pm it |\nabla|) P_N f^\omega \|_{L_t^{q} L_x^p(I\times \rd)} \\
&\lesssim \| \exp(\pm it |\nabla|) P_N f^\omega \|_{L_x^p(\rd)}  +  N |I|^{1-\frac{1}{q}}  \|  \exp(\pm it |\nabla|) P_N f^\omega \|_{L_t^{q} L_x^p(I\times \rd)}
\end{align*}
Taking the \( q\)-th power and averaging over all \( t \in I \), we obtain that 
\begin{equation*}
\| \exp(\pm it_0 |\nabla|) P_N f^\omega \|_{L_x^p(\rd)} \lesssim \left( |I|^{-\frac{1}{q}} + N |I|^{1-\frac{1}{q}} \right) \|  \exp(\pm it |\nabla|) P_N f^\omega \|_{L_t^{q} L_x^p(I\times \rd)}~. 
\end{equation*}
By choosing \( |I|= N^{-1} \), and taking the supremum over all \( t_0 \in [a,b] \), it follows that 
\begin{equation*}
\| \exp(\pm it |\nabla|) P_N f^\omega \|_{L_t^\infty L_x^p([a,b]\times \rd)} \lesssim N^{\frac{1}{q}} \| \exp(\pm it |\nabla|) P_N f^\omega \|_{L_t^{q} L_x^p([a,b]\times \rd)} 
\end{equation*}
\end{proof}

The following estimate, which appeared in the almost sure scattering problem for the nonlinear Schrödinger equation \cite{KMV17}, is useful in combination with Khintchine's inequality. 
\begin{lem}[{Square function estimate \cite[Lemma 2.8]{KMV17}}]\label{prelim:lem_square_function}
Let \( f \in L_x^2(\rd) \) and let \( \varphi \) be as above. Then, it holds that 
\begin{equation}\label{prelim:eq_square_function}
\sum_{k\in \zd} |P_k f(x)|^2 \lesssim (|\widecheck{\varphi}| * |f|^2)(x)~. 
\end{equation}
\end{lem}

In addition to the dyadic decomposition in frequency, we also need a dyadic decomposition in physical space. To avoid confusion, we denote the cut-off function in physical space by \( \chi \). More precisely, we set 
\( \chi_1(x):= \varphi(x) \) and \( \chi_L(x):= \psi(x/L) \), where \( L \geq  2 \). 
\begin{lem}[Mismatch estimates]\label{prelim:lem_mismatch}
Let \( 1 \leq p \leq \infty \) and \( M, N,L \geq 1 \). We further assume the mismatch conditions \( \max(M/N,N/M) \geq 8 \) and \( L \geq 8 \). Then, it holds for all absolute constants \( D  > 0 \) that 
\begin{align}
\| \chi_1 P_M \chi_L \|_{L_x^p \rightarrow L_x^p} \lesssim_D (ML)^{-D} ~ \label{prelim:eq_mismatch_physical}~,\\
\| P_N \chi_1 P_M \|_{L_x^p\rightarrow L_x^p} \lesssim_D (NM)^{-D}~.\label{prelim:eq_mismatch_frequency}
\end{align}
\end{lem}
\begin{proof}
The inequality \eqref{prelim:eq_mismatch_physical} can be found in \cite[Lemma 5.10]{DLM18}. 
An inequality similar to \eqref{prelim:eq_mismatch_frequency} can be found in \cite[Lemma 5.11]{DLM18}, and we present a different argument. \\
Using duality and \( (P_N \chi_1 P_M)^\ast = P_M \chi_1 P_N \), we can assume that \( N \geq M \). From the mismatch condition, it then follows that \( N \geq 8 M \). Thus, we obtain for all \( f \in L_x^p(\rd) \) that 
\begin{equation*}
\| P_N( \chi_1 P_M f) \|_{L_x^p} = \| P_N( (P_{\geq N/8} \chi_1) P_Mf ) \|_{L_x^p} \lesssim \| P_{\geq N/8} \chi_1 \|_{L_x^\infty} \| P_M f \|_{L_x^p} \lesssim N^{-2D} \| f \|_{L_x^p}~. 
\end{equation*}
\end{proof}

The following auxiliary lemma will be helpful in the proof of probabilistic Strichartz estimates. 
\begin{lem}[{\(\ell_{k,l}^2\)-estimate}] \label{prelim:lem_l2}
Let \( s \in \reals \) and let \( f \in H_x^s(\rd) \). For any \( 2 \leq p \leq \infty \), we have that 

\begin{equation}\label{prelim:eq_l2}
\| P_k( \varphi_l f) \|_{\ell_l^2 \ell_k^2 L_x^{p^\prime}(\zd \times \{ k \in \zd\colon \| k \|_\infty \in (N/2,N]\} \times \rd)} \lesssim \| \widetilde{P}_N f \|_{L_x^2(\rd)} + N^{-s-10d} \| f \|_{H_x^s(\rd)}
\end{equation}
\end{lem}

\begin{rem}
The error term \( N^{-s-10d} \| f \|_{H_x^s(\rd)} \) is a result of the non-compact support of \( \widehat{\varphi}_l\), but may essentially be ignored. On a heuristic level, each \( P_k (\varphi_l f) \) is supported on a spatial region of volume \( \sim 1 \), and thus \eqref{prelim:eq_l2} should follow from Hölder's inequality. To make this argument rigorous, we use the square-function estimate and the mismatch estimates above. 
\end{rem}

\begin{proof} Let \( \widetilde{P}_N  \) be the fattened Littlewood-Paley operator as in \eqref{prelim:eq_fattened}. We write \( M \nsim N \) if either \( M< 2^{-10} N \) or \( M > 2^{10} N \). In the following, we implicitly assume that \( \| k \|_\infty \in (N/2,N] \). We then estimate
\begin{equation}\label{prelim:eq_l2_split}
\| P_k( \varphi_l f) \|_{\ell_l^2 \ell_k^2 L_x^{p^\prime}} \leq \| P_k( \varphi_l \widetilde{P}_N f ) \|_{\ell_l^2 \ell_k^2 L_x^{p^\prime}} +  \sum_{M \nsim N } \| P_k( \varphi_l P_M f) \|_{\ell_l^2 \ell_k^2 L_x^{p^\prime}} ~.
\end{equation}
We begin by  controlling the first summand in \eqref{prelim:eq_l2_split}. Using Minkowski's integral inequality and the square-function estimate (Lemma \ref{prelim:lem_square_function}), we obtain that
\begin{align} 
&\| P_k ( \varphi_l \widetilde{P}_N f ) \|_{\ell_l^2 \ell_k^2 L_x^{p^\prime}} 
= \| \varphi_{l^\prime} P_k ( \varphi_l \widetilde{P}_N f ) \|_{\ell_l^2 \ell_k^2 \ell_{{l^\prime}}^1 L_x^{p^\prime}} 
\lesssim \| \varphi_{l^\prime} P_k ( \varphi_l \widetilde{P}_N f ) \|_{\ell_l^2 \ell_k^2 \ell_{{l^\prime}}^1 L_x^{2}} 
\lesssim \| \varphi_{l^\prime} P_k ( \varphi_l \widetilde{P}_N f ) \|_{\ell_l^2 \ell_{{l^\prime}}^1 L_x^{2} \ell_k^2}  \notag \\
&=\Big\| \varphi_{l^\prime} \left( |\widecheck{\varphi}| * |\varphi_l \widetilde{P}_N f|^2 \right)^{\frac{1}{2}} \Big\|_{\ell_l^2 \ell_{l^\prime}^1 L_x^2} ~. \label{prelim:eq_l2_proof1}
\end{align}
Using simple support considerations, we have that
\begin{align*}
&\Big\| \varphi_{l^\prime} \left( |\widecheck{\varphi}| * |\varphi_l \widetilde{P}_N f|^2 \right)^{\frac{1}{2}}\Big \|_{L_x^2}^2 
\leq\Big \| \varphi_{l^\prime} \left( |\widecheck{\varphi}| * |\varphi_l \widetilde{P}_N f|^2 \right) \Big\|_{L_x^1} 
\lesssim \langle l^\prime - l \rangle^{-10d} \| (\varphi_l \widetilde{P}_N f)^2 \|_{L_x^1} \\
&\lesssim \langle l^\prime - l \rangle^{-10d} \| \varphi_l \widetilde{P}_N f\|_{L_x^2}^2~. 
\end{align*}
Inserting this back into \eqref{prelim:eq_l2_proof1}, we obtain that 
\begin{equation*}
\| P_k ( \varphi_l \widetilde{P}_N f ) \|_{\ell_l^2 \ell_k^2 L_x^{p^\prime}} \lesssim   \| \langle l^\prime - l \rangle^{-5d} \varphi_l \widetilde{P}_N f\|_{\ell_l^2 \ell_{l^\prime}^1 L_x^2} \lesssim \| \widetilde{P}_N f \|_{L_x^2} ~.
\end{equation*}
Thus, this yields the first term in \eqref{prelim:eq_l2}. We now control the second summand in \eqref{prelim:eq_l2_split}. First, note that \( P_k = \sum_{2^{-5}N \leq N^\prime \leq 2^5 N} P_{N^\prime} P_k \). Since there exist only \( \sim N^d \) frequencies of magnitude \( \sim N \), we have that 
\begin{align*}
\| P_k ( \varphi_l P_M f ) \|_{\ell_k^2 \ell_l^2 L_x^{p^\prime}} 
&\lesssim N^{\frac{d}{2}} \sum_{2^{-5}N\leq N^\prime \leq 2^5 N} \| P_{N^\prime} (\varphi_l P_M f) \|_{\ell_l^2 L_x^{p^\prime}}~. 
\end{align*}
It now suffices to prove for all \( g \in S(\rd) \), all \( M \nsim N \), and all absolute constants \( D  > 0 \) that 
\begin{equation}\label{prelim:eq_l2_proof2}
\| P_{N^\prime}  ( \varphi_l P_M g) \|_{L_x^{p^\prime}} \lesssim_D (NM)^{-D} \| \langle x-l \rangle^{-D} g \|_{L_x^{2}}~. 
\end{equation}
Using spatial translation invariance, we may set \( l = 0 \). Let \( \{\chi_L\}_{L\geq 1} \) denote the dyadic decomposition in physical space. Using the mismatch estimates (Lemma \ref{prelim:lem_mismatch}), we obtain 
\begin{align*}
&\| P_{N^\prime} (\varphi_0 P_M g) \|_{L_x^{p^\prime}} 
\leq \sum_{L\geq 1} \|P_{N^\prime}  (\varphi_0 P_M  \chi_L g) \|_{L_x^{p^\prime}} 
\leq \sum_{L\geq 1}  \| P_{N^\prime}  \varphi_0 P_M \chi_L \|_{L_x^{p^\prime}\rightarrow L_x^{p^\prime}} \| \widetilde{\chi}_L g \|_{L_x^{p^\prime}} \\
&\lesssim (NM)^{-D} \sum_{L\geq 1} L^{-2D} \| \widetilde{\chi}_L g \|_{L_x^{p^\prime}} 
\lesssim (NM)^{-D} \| \langle x \rangle^{-D} g \|_{L_x^2}~. 
\end{align*}
\end{proof}

As a direction consequence of \eqref{prelim:lem_l2}, we also obtain the following estimate on the \( H^s\)-norm of the microlocal randomization.
\begin{lem}[{\(H_x^s\)-norm of \( f^\omega \)}]\label{prelim:lem_random_hs}
Let \( f \in H_x^s(\rd) \) and let \( f^\omega \) be its microlocal randomization. We further set 
\begin{equation}\label{prelim:eq_frequency_decomposition}
f_1^\omega := \sum_{l\in \zd}  X_{0,l} P_0(\varphi_l f) \qquad \text{and} \qquad f_N^\omega := \sum_{\substack{k,l\in \zd\\ \| k\|_\infty \in (N/2,N]}}  \hspace{-2ex} X_{k,l} P_k(\varphi_l f)~, \quad \text{where} ~N \geq 2~.  
\end{equation}
Then, we have for all \(  2 \leq r < \infty \) that 
\begin{equation}\label{prelim:eq_random_hs}
 \| f^\omega \|_{L_\omega^r H_x^s} \simeq    \| N^s f_N^\omega \|_{L_\omega^r \ell_N^2 L_x^2} \lesssim  \sqrt{r}\| f \|_{H_x^s}~. 
\end{equation}
\end{lem}
\begin{proof}
The first equivalence in \eqref{prelim:eq_random_hs} is a direct consequence of the definition of the \( H_x^s\)-norm. Now, we prove the bound in \eqref{prelim:eq_random_hs}. From Minkowski's integral inequality, Khintchine's inequality, and Lemma \ref{prelim:lem_l2}, we have for all \( N \geq 2 \) that 
\begin{align*}
 \| N^s f_N^\omega \|_{L_\omega^r  L_x^2} &\leq \| N^s  \hspace{-2ex} \sum_{ \| k \|_\infty \in (N/2,N]}  \hspace{-2ex} X_{k,l} P_k(\varphi_l f ) \|_{L_x^2 L_\omega^r} \\
 &\leq \sqrt{r} N^s \| P_k(\varphi_l f)\|_{L_x^2 \ell_{k,l}^2(\|k\|_\infty \in (N/2,N]} \\
 &\lesssim \sqrt{r} \left( N^s \| \widetilde{P}_N f \|_{L_x^2(\rd)} + N^{-10d} \| f \|_{H_x^s(\rd)} \right)
\end{align*}
The same argument also applies to \( N=1 \). After taking the \( \ell_N^2\)-norm, this completes the proof. 
\end{proof}

\subsection{Strichartz estimates}
The individual blocks in the microlocal randomization or the Wiener randomization have frequency support inside a unit-sized cube (at a large distance from the origin). Since this rules out the Knapp example, one expects a refined dispersive estimate. The following lemma is due to Klainerman and Tataru \cite{KT99}, and it has first been used in the probabilistic context by \cite{DLM17}. 

\begin{lem}[{Refined dispersive estimate by Klainerman-Tataru \cite{KT99}}]\label{prelim:lem_refined_dispersive}
Let \( f \in L^1(\reals^d) \), let \( k \in \mathbb{Z}^d \) satisfy \( \| k \|_\infty \in (N/2,N]\), and let \( M \leq N \).  Then it holds for all \( t \in \reals \) and  \( 2 \leq p \leq \infty \) that 
\begin{equation}\label{prelim:eq_refined_dispersive}
\| \Htpm P_{M;k} f \|_{L_x^p(\mathbb{R}^d)} \lesssim \frac{M^{d (1-\frac{2}{p})}}{(1+\frac{M^2}{N} |t|)^{(d-1)(\frac{1}{2}-\frac{1}{p})}} \| f \|_{L_x^{p^\prime}(\reals^d)}~. 
\end{equation}
\end{lem}
As stated, the inequality \eqref{prelim:eq_refined_dispersive} essentially follows from \cite{KT99}. For the sake of completeness, we present the modification below. 
\begin{proof} By interpolation against the energy estimate \( \| \Htpm P_{M;k} f \|_{L_x^2(\mathbb{R}^d)} \leq \| f\|_{L_x^2(\reals^d)} \), it suffices to prove \eqref{prelim:eq_refined_dispersive} for \( p=\infty \). 
The inequality \cite[(A.66)]{KT99}, where \( \mu = M/N \),  and a scaling argument yield
\begin{equation}\label{prelim:eq_klainerman_tataru}
\| \Htpm  P_{M;k} f \|_{L_x^\infty(\mathbb{R}^d)} \lesssim  \frac{MN^{d-1}}{(1+N|t|)^{\frac{d-1}{2}}} \| f \|_{L_x^{1}(\reals^d)}~. 
\end{equation}
We now distinguish two cases. If \( |t|\lesssim N/M^2 \), then Bernstein's inequality (Lemma \ref{prelim:lem_bernstein}) yields that 
\begin{equation*}
\| \Htpm  P_{M;k} f \|_{L_x^\infty(\mathbb{R}^d)}  \lesssim  M^{\frac{d}{2}} \| \Htpm  P_{M;k} f \|_{L_x^2(\mathbb{R}^d)} = M^{\frac{d}{2}} \|  P_{M;k} f \|_{L_x^2(\mathbb{R}^d)} 
\lesssim  M^{d} \|  P_{M;k} f \|_{L_x^1(\mathbb{R}^d)} ~. 
\end{equation*}
If \( |t|\lesssim N/M^2 \), then \eqref{prelim:eq_klainerman_tataru} yields that 
\begin{equation*}
\| \Htpm  P_{M;k} f \|_{L_x^\infty(\mathbb{R}^d)}\lesssim \frac{MN^{d-1}}{(N|t|)^{\frac{d-1}{2}}} \| f \|_{L_x^1(\reals^d)}  = \frac{M^d}{(\frac{M^2}{N}t)^{\frac{d-1}{2}}} \| f \|_{L_x^1(\reals^d)} \lesssim \frac{M^d}{(1+ \frac{M^2}{N}t)^{\frac{d-1}{2}}} \| f \|_{L_x^1(\reals^d)} ~. 
\end{equation*}
\end{proof}

In this paper, we are mainly concerned with the case \( M= 1 \). Then, \eqref{prelim:eq_refined_dispersive} describes the linear evolution on short time intervals more accurately than  \eqref{prelim:eq_klainerman_tataru}. 
As a corollary of the refined dispersive estimate, we obtain the following refined Strichartz estimate. \\
Let \( 2 \leq q,p \leq \infty \). We call the pair \( (q,p) \) wave-admissible if 
\begin{equation*}
\frac{1}{q} + \frac{d-1}{2p} \leq \frac{d-1}{4} \qquad \text{and} \qquad (q,p,d) \neq (2,\infty,3)~. 
\end{equation*}

\begin{cor}[{Refined Strichartz estimates \cite{KT99}}]\label{prelim:cor_refined_strichartz}
Let \( f \in L_x^2(\rd) \), let \( k \in \mathbb{Z}^d \) satisfy \( \| k \|_\infty \in (N/2,N]\), and let \( M\leq N \). Then, we have for all wave-admissible pairs \( (q,p) \) that 
\begin{equation}
\| \exp(\pm it|\nabla|) P_{M;k} f \|_{L_t^q L_x^p(\reals\times \rd)} \lesssim M^{\frac{d}{2}-\frac{2}{q}-\frac{d}{p}} N^{\frac{1}{q}} \| P_{M;k} f\|_{L_x^2(\rd)}
\end{equation}
\end{cor}
The derivation of the refined Strichartz estimate from Lemma \ref{prelim:lem_refined_dispersive} follows from a standard \( TT^\ast \)-argument, and we therefore omit the proof. For the endpoint \( (2,2 (d-1)/(d-3)) \), we also refer to \cite{KT98}. Let us emphasize two special cases: If \( M= N \), we obtain the usual scaling factor \( N^{\frac{d}{2}-\frac{1}{q}-\frac{d}{p}} \), and if \( M=1 \), we obtain the factor \( N^{\frac{1}{q}} \), which does not depend on \( p \).

\section{Probabilistic Strichartz estimates}

In this section, we derive probabilistic Strichartz estimates (cf. \cite{BOP2014,DLM17,LM13}) and a probabilistic long-time decay estimate (cf. \cite{Murphy17}). To keep the exposition self-contained, we include the (short) proofs. Recall from \eqref{prelim:eq_frequency_decomposition} that
\begin{equation*}
f_1^\omega := \sum_{l\in \zd}  X_{0,l} P_0(\varphi_l f) \qquad \text{and} \qquad f_N^\omega := \sum_{\substack{k,l\in \zd\\ \| k\|_\infty \in (N/2,N]}}  \hspace{-2ex} X_{k,l} P_k(\varphi_l f)~, \quad \text{where} ~N \geq 2~.  
\end{equation*}

\begin{lem}[{Probabilistic Strichartz estimate}]\label{prob:lem_probabilistic_strichartz}
Let \( f \in H_x^s(\rd) \) and let \( f^\omega \) be its microlocal randomization. Then, it holds for all \( N \geq 1 \), all wave-admissible exponent pairs \( (q,p) \), and all \( 1 \leq r < \infty \) that 
\begin{equation}\label{prob:eq_probabilistic_strichartz}
\| \exp(\pm it |\nabla|) f_N^\omega \|_{L_\omega^r L_t^q L_x^p(\Omega\times \reals \times \rd)} \lesssim \sqrt{r} N^{\frac{1}{q}-s+} \| f \|_{H_x^s(\rd)}~. 
\end{equation}
\end{lem}
This estimate has previously appeared for the Wiener randomization in \cite{DLM17}. 

\begin{proof} In the following, we implicitly assume that \( k \in \zd \) always satisfies \( \| k \|_\infty \in (N/2,N] \). 
First, we assume that \( 2 \leq p,q < \infty \), and that \( \max(p,q)\leq r < \infty\).  Using Minkowski's integral inequality (Lemma \ref{prelim:lem_minkowski}), Khintchine's inequality (Lemma \ref{prelim:lem_khintchine}), and the refined dispersive estimate (Lemma \ref{prelim:lem_refined_dispersive}), we have that 
\begin{align*}
&\| \exp(\pm it |\nabla|) f_N^\omega \|_{L_\omega^r L_t^q L_x^p} 
\leq \| \exp(\pm it |\nabla|) f_N^\omega \|_{L_t^q L_x^p L_\omega^r} 
\lesssim \sqrt{r} \| \exp(\pm it |\nabla|) P_k(\varphi_l f) \|_{L_t^q L_x^p \ell_{k,l}^2} \\
&\leq \sqrt{r} \| \exp(\pm it |\nabla|) P_k(\varphi_l f) \|_{\ell_{k,l}^2 L_t^q L_x^p} 
\lesssim \sqrt{r} N^{\frac{1}{q}}  \| P_k(\varphi_l f) \|_{\ell_{k,l}^2 L_x^2} 
\lesssim \sqrt{r} N^{\frac{1}{q}-s} \| f \|_{H_x^s}~. 
\end{align*}
In the last inequality, we have also used Lemma \ref{prelim:lem_l2}. The estimate for \( 1 \leq r \leq \max(p,q) \) then follows from Hölder's inequality. Thus, it remains to treat the cases \( q=\infty \) and/or \( p=\infty \). This is a know technical issue, see \cite[Remark 3.8]{Bringmann18} for a discussion. Both cases can be reduced to the previous estimate by using Lemma \ref{prelim:lem_inf_removal} and Bernstein's inequality.  
\end{proof}

\begin{lem}[Probabilistic long-time decay]\label{prob:lem_long_time_decay}
Let \( f \in L_x^2(\mathbb{R}^d) \) and let \( f^\omega \) be its microlocal randomization. Furthermore, let \( 1 \leq q < \infty \) and \( 2 \leq p \leq \infty \) be such that 
\begin{equation}\label{prob:eq_condition}
\frac{1}{q} + \frac{d-1}{p} < \frac{d-1}{2}~. 
\end{equation}
Then, we have for all \( 1 \leq r < \infty \) that
\begin{equation}\label{prob:eq_long_time_decay}
\| \Htpm f_N^\omega \|_{L_\omega^r L_t^q L_x^p( \Omega \times [T,\infty) \times \reals^d)} \lesssim_{q,p} \sqrt{r} N^{\frac{1}{q}-s+} \left( 1 + \frac{T}{N} \right)^{\frac{1}{q}+\frac{d-1}{p}-\frac{d-1}{2}} \| f\|_{H_x^s(\rd)}. 
\end{equation}
\end{lem}

Lemma \ref{prob:lem_long_time_decay} has previously been used for a physical space randomization in \cite[Proposition 3.1]{Murphy17}. In contrast to the standard Strichartz estimates, which are time-translation invariant, \eqref{prob:eq_long_time_decay} provides a quantitative decay rate. The motivation behind this estimate is illustrated in Figure \ref{figure:physical}. In this paper, we only require the following special case. 
\begin{cor}\label{prob:cor_long_time_decay}
Let \( f \in L_x^2(\mathbb{R}^4) \) and let \( f^\omega \) be its microlocal randomization. Then, we have for all \( \theta >0 \) that 
\begin{equation}
\| \Htpm f_N^\omega \|_{L_\omega^r L_t^1 L_x^\infty(\Omega \times [N^{1+\theta},\infty) \times \reals^4)} \lesssim \sqrt{r} N^{1-\frac{\theta}{2}+} \| f_N \|_{L_x^2(\rfour)}~.
\end{equation}
\end{cor}

\begin{figure}[t!]
\begin{center}
\includegraphics[height=5cm]{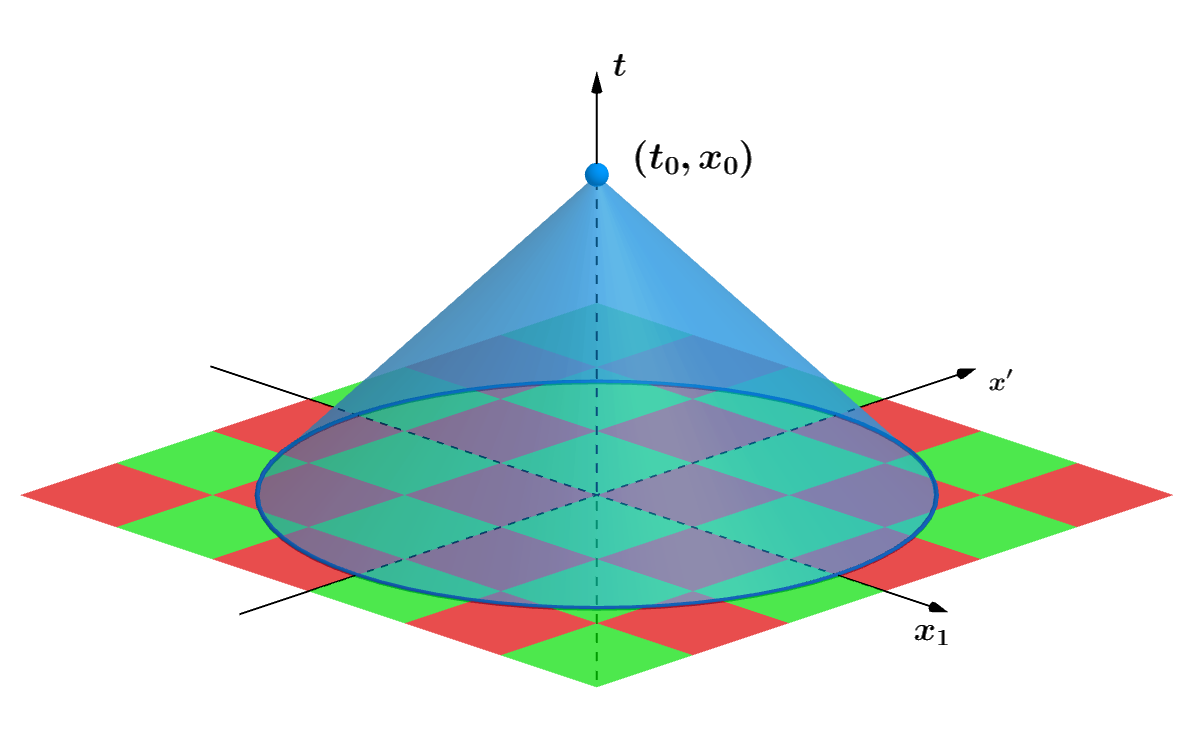}
\end{center}
\caption*{\small{This figure illustrates the effect of the physical randomization on the linear evolution. At the point \( (t_0,x_0) \), the linear evolution depends on the initial data in a large region of space. Due to the physical randomization, the initial data in different spatial regions cannot constructively interfere, and hence we expect an improved decay.  }}
\caption{Effect of physical randomization}
\label{figure:physical}
\end{figure}

\begin{rem}
Due to \eqref{prob:eq_condition}, the \( L_t^1 L_x^\infty \)-estimate fails logarithmically in three dimensions. 
\end{rem}

\begin{proof}[Proof of Lemma \ref{prob:lem_long_time_decay}]
We essentially follow the argument in \cite{Murphy17}. Let us first assume that \( 2\leq q,p < \infty \). \\
We further assume that \( r \geq \max(q,p) \), the corresponding estimate for \( 1\leq r < \max(q,p) \) then follows from Hölder's inequality. Using Minkowski's integral inequality (Lemma \ref{prelim:lem_minkowski}), Khintchine's inequality (Lemma \ref{prelim:lem_khintchine}), and the refined dispersive estimate (Lemma \ref{prelim:lem_refined_dispersive}), we have that 
\begin{align*}
&\| \Htpm f_N^\omega \|_{L_\omega^r L_t^q L_x^p( \Omega \times [T,\infty) \times \reals^d)} \displaybreak[1] \\
&\leq \| \Htpm f_N^\omega  \|_{L_t^q L_x^p L_\omega^r(  [T,\infty) \times \reals^d \times \Omega)} \displaybreak[1]\\
&\lesssim \sqrt{r} \| \Htpm P_k(\varphi_l f) \|_{L_t^q L_x^p \ell_{k,l}^2([T,\infty)\times \reals^d \times \mathbb{Z}^{d+d})} \displaybreak[1]\\
&\lesssim \sqrt{r} \| \Htpm P_k(\varphi_l f) \|_{\ell_{k,l}^2 L_t^q L_x^p (\mathbb{Z}^{d+d} \times [T,\infty)\times \reals^d )}\displaybreak[1]\\
&\lesssim \sqrt{r} \left\| \Big( 1+ \frac{|t|}{N} \Big)^{-(d-1) (\frac{1}{2}-\frac{1}{p})} \| P_k(\varphi_l f) \|_{L_x^{p^\prime}} \right\|_{\ell_{k,l}^2 L_t^q(\mathbb{Z}^{d+d} \times [T,\infty))}\displaybreak[1] \\
&\lesssim \sqrt{r} \left\|\Big( 1+ \frac{|t|}{N} \Big)^{-(d-1) (\frac{1}{2}-\frac{1}{p})}  \right\|_{L_t^q([T,\infty))} \| P_k(\varphi_l f) \|_{\ell_{k,l}^2 L_x^{p^\prime}(\mathbb{Z}^{d+d}\times \reals^d)}~. 
\end{align*}
Using condition \eqref{prob:eq_condition}, we obtain
\begin{equation*}
 \left\|\Big( 1+ \frac{|t|}{N} \Big)^{-(d-1) (\frac{1}{2}-\frac{1}{p})}  \right\|_{L_t^q([T,\infty))} \lesssim N^{\frac{1}{q}} \left( 1 + \frac{T}{N} \right)^{\frac{1}{q}+\frac{d-1}{p}-\frac{d-1}{2}}
 \end{equation*}
 Finally, from Lemma \ref{prelim:lem_l2} we have that 
 \begin{equation*}
 \| P_k(\varphi_l f) \|_{\ell_{k,l}^2 L_x^{p^\prime}(\mathbb{Z}^{d+d}\times \reals^d)} \lesssim N^{-s} \| f\|_{H_x^s}~. 
 \end{equation*}
 This finishes the proof in the case \( 2 \leq q,p < \infty \). Using Bernstein's inequality, we can reduce the case \( p = \infty \) to \( p< \infty \).  Thus, it remains to treat the range \( 1 \leq q < 2 \). Using a dyadic decomposition in time, we have for all \( T \geq N \) that 
 \begin{align*}
 &\| \Htpm f_N^\omega \|_{L_\omega^r L_t^q L_x^p( \Omega \times [T,\infty) \times \reals^d)}\displaybreak[1] \\
 &\leq \sum_{m=0}^{\infty} \| \Htpm f_N^\omega \|_{L_\omega^r L_t^q L_x^p( \Omega \times [2^m T,2^{m+1}T) \times \reals^d)} \displaybreak[1]\\
 &\leq \sum_{m=0}^{\infty}  (2^{m} T)^{\frac{1}{q}-\frac{1}{2}} \| \Htpm f_N^\omega \|_{L_\omega^r L_t^2 L_x^p( \Omega \times [2^m T,2^{m+1}T) \times \reals^d)}\displaybreak[1]\\
 &\lesssim \sum_{m=0}^\infty (2^{m} T)^{\frac{1}{q}-\frac{1}{2}}  N^{\frac{1}{2}-s+} \left( \frac{2^m T}{N} \right)^{\frac{1}{2}+\frac{d-1}{p}-\frac{d-1}{2}}\| f\|_{H_x^s}\displaybreak[1]\\
 &\lesssim N^{\frac{1}{q}-s+} \left(\frac{T}{N} \right)^{\frac{1}{q}+\frac{d-1}{p}-\frac{d-1}{2}}\| f\|_{H_x^s}~. 
 \end{align*}
 In the second last line, we used condition \eqref{prob:eq_condition}. For \( T \leq N \), we also have that 
 \begin{align*}
  \| \Htpm f_N^\omega \|_{L_\omega^r L_t^q L_x^p( \Omega \times [T,N) \times \reals^d)}  
  \lesssim N^{\frac{1}{q}- \frac{1}{2}} \| \Htpm f_N^\omega \|_{L_\omega^r L_t^2 L_x^p( \Omega \times [0,N) \times \reals^d)} 
  \lesssim  N^{\frac{1}{q}-s+}\| f\|_{H_x^s}~. 
 \end{align*}

\end{proof}

\begin{definition}[Auxiliary norm]\label{prop:def_auxiliary}
Let \( 0\leq s < 1 \), let \( (f_0,f_1)\in H^s(\rfour)\times H^{s-1}(\rfour) \), and let \( N_0\geq 1\).  We then define
\begin{align*}
\| (f_0,f_1) \|_{Z(N_0)} &:= \sum_{N\geq N_0} N^{s+\frac{\theta}{2}-1-\delta} \left \| \cos(t|\nabla|)  f_{0,N} + \frac{\sin(t|\nabla|)}{|\nabla|}  f_{1,N} \right \|_{L_t^1 L_x^\infty([N^{1+\theta},\infty)\times \rfour)}\\
	&~+\sum_{N\geq N_0} N^{s-\delta}  \left \| \cos(t|\nabla|) f_{0,N} + \frac{\sin(t|\nabla|)}{|\nabla|} f_{1,N} \right \|_{L_t^\infty L_x^\infty([0,\infty)\times \rfour)}
\end{align*}
\end{definition}

From Proposition \ref{prob:lem_probabilistic_strichartz} and Corollary \ref{prob:cor_long_time_decay}, it follows that 
\begin{equation*}
\| (f_0^\omega,f_1^\omega) \|_{L_\omega^r Z(1)} \lesssim \sqrt{r} \| (f_0,f_1) \|_{H_x^s\times H_x^{s-1}}~. 
\end{equation*}

\section{Wave packet decomposition}\label{section:wp}
In this section, we use a wave packet decomposition to better understand the (random) linear evolution. \\
This part of the argument does not rely on the additional randomization in physical space. We therefore phrase all results in a way that applies to both the microlocal and the Wiener randomization, and hope that this 
facilitates future applications. With this in mind, we now rewrite the microlocal randomization in a form that resembles the Wiener randomization. \\
Let the random variables \( \{ X_{k,l} \}_{k,l \in \zd} \) be as in Definition \ref{intro:def_microlocal}, let \( \{ \epsilon_k \}_{k\in I\cup \{ 0 \}} \) be a family of independent random signs, and set \( \epsilon_{-k} = \epsilon_k \) for all \( k \in I \). We can then define \( Y_{k,l} := \epsilon_k X_{k,l} \). For a sequence of multi-indices \( k,l_1,\hdots,l_J\in \zd \) and any sequence of Borel-measurable sets \( A_1,\hdots,A_J \subseteq \reals\), we have that 
\begin{align*}
\mathbb{P}(\epsilon_k =1 , Y_{k,l_1}\in A_1,\hdots,Y_{k,l_J} \in A_J) =\mathbb{P}(\epsilon_k =1 , X_{k,l_1}\in A_1,\hdots,X_{k,l_J} \in A_J)   \\
= \mathbb{P}(\epsilon_k =1) \prod_{j=1}^J \mathbb{P}(X_{k,l_j}\in A_j)= \mathbb{P}(\epsilon_k =1) \prod_{j=1}^J \mathbb{P}(Y_{k,l_j}\in A_j) ~.
\end{align*}
In the last equality, we have used that the random variables \( X_{k,l} \) are symmetric. Therefore, for a fixed \( k \in \zd \), the family \( \{ \epsilon_k \} ~ \mathsmaller{\bigcup} ~ \{ Y_{k,l} \}_{l\in \zd} \) is independent.  
From this, it then easily follows that the whole family \( \{ \epsilon_k \}_{k\in I \cup \{ 0 \}} ~ \mathsmaller{\bigcup} ~ \{ Y_{k,l} \}_{k \in I \cup \{ 0 \}, l\in \zd} \) is independent. We then rewrite the microlocal randomization as 
\begin{equation}\label{wp:eq_microlocal_to_wiener}
f^\omega = \sum_{k,l\in \zd} X_{k,l} P_k(\varphi_l f) = \sum_{k\in \zd} \epsilon_k P_k( \sum_{l\in \zd} Y_{k,l} \varphi_l f ) =\sum_{k\in \zd} \epsilon_k f_k, \qquad \text{where} \quad f_k := P_k  (\sum_{l\in \zd} Y_{k,l} \varphi_l f)~.
\end{equation}
Due to the independence properties discussed above, we can regard the functions \( \{ f_k \} \) as deterministic by conditioning on the random variables \( \{ Y_{k,l} \}_{k,l} \), and only utilize the randomness through the random signs \( \{ \epsilon_k \}_k \). Note that \eqref{wp:eq_microlocal_to_wiener} closely resembles the Wiener randomization.  \\
To motivate the wave packet decomposition below, we now rewrite the linear evolution with initial data \( (f_0^\omega,f_1^\omega) \). Using the notation from \eqref{wp:eq_microlocal_to_wiener}, we first introduce the half-wave operators by writing
\begin{equation}\label{wp:eq_half_wave_decomposition}
\begin{aligned}
&\cos(t|\nabla|) f_0^\omega + \frac{\sin(t|\nabla|)}{|\nabla|} f_1^\omega \\
&= \sum_{k\in \zd} \epsilon_k \left(  \cos(t|\nabla|) f_{0;k}^\omega + \frac{\sin(t|\nabla|)}{|\nabla|} f_{1;k}^\omega \right) \\
&= \sum_{k\in \zd} \epsilon_k \left [ \exp(it|\nabla|) \bigg( \frac{f_{0;k}+ i |\nabla|^{-1} f_{1;k}}{2} \bigg) + \exp(-it|\nabla|) \bigg( \frac{f_{0;k}- i |\nabla|^{-1} f_{1;k}}{2} \bigg) \right] \\
&=: \sum_{k\in \zd} \epsilon_k \left[ \exp(it|\nabla|) f_k^+ + \exp(-it|\nabla|) f_k^- \right]~. 
\end{aligned}
\end{equation}
As in \eqref{prelim:eq_frequency_decomposition}, we also decompose dyadically in frequency space, and write 
\begin{equation}\label{wp:eq_FN}
F_N^{\pm}:=  \sum_{\| k\|_\infty \in (N/2,N]} \epsilon_k \exp(\pm it|\nabla|) f_{k}^\pm \qquad \text{and} \qquad F_N := F_N^{+} + F_N^{-}~. 
\end{equation}
Let \( k \in \mathbb{Z}^d \) with \( \| k\|_\infty \in (N/2,N] \), and let \( l \in \mathbb{Z}^d \). We define the tubes \( T_{k,l}^\pm \) by
\begin{equation}
T_{k,l}^+ := \{ (t,x)\in [0,N] \times \reals^d \colon \| x- ( l \mp t \cdot k/\|k\|_2 ) \|_2 \leq 1 \}~.  \label{wp:eq_tube}
\end{equation}
Here, the superscripts \( \pm \) are chosen so that the tubes correspond to the operators \( \Htpm \). The dimensions and the shape of the tubes are illustrated in the introduction, see Figure \ref{figure:single_wave_packet}. Motivated by the Doppler effect, the tubes \( T_{k,l}^+ \) are sometimes called red tubes, and the tubes \( T_{k,l}^- \) are sometimes called blue tubes. 

\begin{prop}[Spatial wave packet decomposition]\label{wp:prop_wave_packet_decomposition}
Let \( k \in \mathbb{Z}^d \) with \( \| k\|_\infty \in (N/2,N] \). Let \(f_k \in \lxtwo \) be a function such that \( \supp \widehat{f_k}\subseteq k + [-1,1]^d \). Then, there exists a decomposition 
\begin{equation*}
f_k = \sum_{l\in \mathbb{Z}^d} f_{k,l}
\end{equation*}
such that
\begin{enumerate}
\item \label{wp:item_supp}  \( \supp \widehat{f_{k,l}}\subseteq k+[-4,4]^d \) for all \( l \in \mathbb{Z^d} \), 
\item\label{wp:item_orthogonality} the family \( \{ f_{k,l} \}_{l\in \zd} \) satisfies the almost-orthogonality condition\begin{equation}\label{wp:eq_orthogonality}
\sum_{l\in \mathbb{Z}^d} \| f_{k,l} \|_{\lxtwo}^2 \lesssim \| f \|_{\lxtwo}^2~, 
\end{equation}
\item \label{wp:item_decay} and for any \( D \geq 1 \), any \( l \in \mathbb{Z}^d \), and all \( (t,x)\in [0,N]\times \reals^d \), it holds that 
\begin{equation}\label{wp:eq_decay}
| \exp(\pm it|\nabla|) f_{k,l}(x) |\lesssim_D (1+ \dist((t,x), T_{k,l}^{\pm}) )^{-D} \| f_k \|_{\lxtwo} ~.  
\end{equation}
\end{enumerate}
\end{prop}

Wave packet decomposition as in Proposition \ref{wp:prop_wave_packet_decomposition} have been used extensively in the literature, see e.g. \cite{Bourgain91,Cordoba77,Fefferman73,Guth16,Wolff01} and the survey \cite{Tao04b}. We present the details below, but encourage the expert reader to skip ahead to the end of the proof. 

\begin{proof}
We define the fattened projection
\begin{equation*}
\widetilde{P}_k := \sum_{\|k^\prime -k\|_\infty \leq 2} P_{k^\prime}~. 
\end{equation*}
Then, it holds that 
\begin{equation*}
f= \widetilde{P}_k f = \sum_{l\in \mathbb{Z}^d} \widetilde{P}_k( \varphi_l f ) =: \sum_{l\in \mathbb{Z}^d} f_{k,l}~. 
\end{equation*}
The frequency support condition \ref{wp:item_supp} directly follows from the definition of \( \widetilde{P}_k \). Furthermore, the almost orthogonality \ref{wp:item_orthogonality} follows from
\begin{equation*}
\sum_{l\in \zd } \| f_{k,l} \|_{\lxtwo}^2 \lesssim \sum_{l\in \zd} \| \varphi_l f_k\|_{\lxtwo}^2 \lesssim \| f_k \|_{\lxtwo}^2~. 
\end{equation*}
Thus, it remains to prove the decay estimate \ref{wp:item_decay}. We only treat the operator \( \exp(it|\nabla|) \), since the proof for \( \exp(-it|\nabla|) \) is similar. If \( N \lesssim 1 \), the estimate is trivial.  Thus, we may assume that \( N \gg 1 \). The argument is based on the method of non-stationary phase. For all \( t\in [0,N] \) and \( x \in \rd \), we have that 
\begin{align*}
&\exp(it|\nabla|) f_{k,l}(x) \\
&= \frac{1}{(2\pi)^{\frac{d}{2}}} \int_{\rd} \exp(ix\xi+ i t |\xi|) \widetilde{\psi}(\xi-k) (\widehat{\varphi}_l * \widehat{f})(\xi) \dxi \\
&= \frac{1}{(2\pi)^d} \int_{\rd} \int_{\rd} \int_{\rd} \exp(ix\xi+ i t |\xi|) \widetilde{\psi}(\xi-k) \exp(-iy(\xi-\eta)) \varphi(y-l) \widehat{f}(\eta) \dy \deta \dxi \\
&= \frac{1}{(2\pi)^d} \exp(ixk+it|k|) \int_{\rd} K(\eta;t,x) \exp(-il(k-\eta)) \widehat{f}(\eta) \deta~,
\end{align*}
where the kernel \( K(\eta;t,x) \) is given by 
\begin{equation*}
K(\eta;t,x):= \int_\rd \int_\rd  \exp(i(x-l) \xi+ i t (|\xi+k|-|k|))  \exp(-iy(\xi+k-\eta)) \widetilde{\psi}(\xi) \varphi(y) \dxi \dy ~. 
\end{equation*}
Since \( \supp \varphi \subseteq [-1,1]^d \), the function \( a(\xi;y,\eta):= \exp(-iy(\xi+k-\eta)) \widetilde{\psi}(\xi) \varphi(y) \) has uniformly bounded derivatives in \( \xi \), i.e, we have for all \( \alpha \in \mathbb{N}_0^d \) that \( |\partial_\xi^\alpha a(\xi;y,\eta)|\lesssim_\alpha 1 \). Using the support conditions in the variables \( y\) and \( \eta \), it thus suffices to prove for all \( a \in C^\infty_c([-2,2]^d) \) that
\begin{equation}\label{wp:eq_non_stationary}
\left| \int_{\rd} \exp(i(x-l) \xi+ i t (|\xi+k|-|k|))  a(\xi) \dxi \right| \lesssim_M ( 1+ | x - l + t \cdot k/\| k\|_2 )^{-M} ~. 
\end{equation}
Due to the compact support of \( a(\xi) \), we restrict to \( |\xi|\leq 2 \). The bound for \( {|x-l + t \cdot k/\|k\|_2 |\lesssim 1}\)  is trivial. Thus, we may assume that \({ |x-l+ t \cdot k/\| k\|_2 |\gg 1 }\). We define the phase function 
\begin{equation*}
\Phi_k(\xi)= \Phi_k(\xi;t,x,l)= i(x-l) \xi+ i t (|\xi+k|-|k|)~. 
\end{equation*}
Then, we have that 
\begin{align*}
\nabla_\xi \Phi_k(\xi)  
&=\nabla_\xi  \Big( (x-l) \xi + t ( |\xi+k| - |k|) \Big) \\
&= x-l + t \frac{\xi+k}{|\xi+k|} \\
&= x-l + t \frac{k}{|k|} + t \Big( \frac{|k|-|\xi-k|}{| \xi+k| |k|}~  k + \frac{\xi}{|\xi+k|} \Big) ~. 
\end{align*}
From the assumption \( |t|\leq N \), it follows that \( \nabla_\xi \Phi_k = x-l + t k/|k| + \mathcal{O}(1) \). We also write \begin{equation*}
 \nabla_\xi \Phi_k (\xi)=  x-l + t \frac{k}{|k|} + t \Psi_k \big(\frac{\xi}{|k|}\big) \frac{k}{|k|}~,
 \end{equation*}
 where 
\begin{equation*}
\Psi_k(\nu):= \frac{1-|\nu-k/|k||}{|\nu+k/|k||} + \frac{\nu}{|k/|k|+\nu|}~.
\end{equation*}
From rotation invariance, it follows easily that \( |\nabla^\alpha \Psi_k(\nu)| \lesssim_\alpha 1\) for all \( |\nu|\leq 1/10 \), uniformly in \( k \). This leads to \begin{equation*}
|\nabla^\alpha \Phi_k(\xi)|\lesssim_\alpha \frac{|t|}{|k|^{|\alpha|-1}} \lesssim 1 \qquad \text{for all } |\alpha|\geq 2~.  
\end{equation*}
We then rewrite the integral in \eqref{wp:eq_non_stationary} as 
\begin{align*}
&\int_{\rd} \exp(i(x-l) \xi+ i t (|\xi+k|-|k|))  a(\xi) \dxi \\
&= \int_{\rd} \left( \left( \frac{(-i) \nabla_\xi \Phi_k(\xi) \nabla_\xi}{|\nabla \Phi_k(\xi)|^2} \right)^{M} \exp(i\Phi_k(\xi)) \right) \cdot a(\xi)\dxi \\
&= \int_{\rd} \exp(i\Phi_k(\xi))  \left( \nabla_\xi \cdot \frac{i \nabla_\xi \Phi_k(\xi)}{|\nabla \Phi_k(\xi)|^2} \right)^{M} a(\xi) \dxi~. 
\end{align*}
The inequality \eqref{wp:eq_non_stationary} then follows from the bounds on the phase function above. 

\end{proof}

The wave packet decomposition in Proposition \ref{wp:prop_wave_packet_decomposition} is valid on the time interval \( [0,N] \), and the physical localization deteriorates for larger times. When analyzing the linear evolution on an interval of the form \( [t_0,t_0+N) \), with \( t_0\in N\mathbb{N}_0 \), we therefore use the wave packet decomposition of \( \exp(\pm it_0 |\nabla|) f_k \). To state the result, we set
\begin{equation*}
T_{k,l;t_0}^{\pm} := \{ (t,x) \in [t_0,t_0+N]\times \rd \colon  \| x- ( l \mp (t-t_0) \cdot k/\|k\|_2 ) \|_2 \leq 1 \}~.  
\end{equation*}

\begin{cor}[Time-translated spatial wave packet decomposition]\label{wp:cor_wave_packet_decomposition}
Let \( k \in \mathbb{Z}^d \) with \( \| k\|_\infty \in (N/2,N] \), and let \( t_0 \in N \mathbb{N}_0 \). Let \(f_k \in \lxtwo \) be a function satisfying \( \supp \widehat{f_k}\subseteq k + [-1,1]^d \). Then, there exists a decomposition 
\begin{equation*}
\exp( \pm i t_0 |\nabla|) f_k = \sum_{l\in \mathbb{Z}^d} f_{k,l;t_0}^\pm
\end{equation*}
such that
\begin{enumerate}
\item  \( \supp \widehat{f_{k,l;t_0}^\pm}\subseteq k+[-2,2]^d \) for all \( l \in \mathbb{Z^d} \), 
\item the family \( \{ f_{k,l;t_0}^\pm \}_{l\in \zd} \) satisfies the almost-orthogonality condition\begin{equation}\label{wp:eq_orthogonality}
\sum_{l\in \mathbb{Z}^d} \| f_{k,l;t_0}^\pm \|_{\lxtwo}^2 \lesssim \| f \|_{\lxtwo}^2~, 
\end{equation}
\item and for any \( D \geq 1 \), any \( l \in \mathbb{Z}^d \), and all \( (t,x)\in [t_0,t_0+N]\times \reals^d \), it holds that 
\begin{equation}\label{wp:eq_decay}
| \exp(\pm i(t-t_0)|\nabla|) f_{k,l;t_0}^\pm(x) |\lesssim_D (1+ \dist((t,x), T_{k,l;t_0}^{\pm}) )^{-D} \| f_k \|_{\lxtwo} ~.  
\end{equation}
\end{enumerate}
\end{cor}
\begin{proof}
We apply Proposition \ref{wp:prop_wave_packet_decomposition} to \( \exp(\pm it_0 |\nabla|) f_k \). 
\end{proof}

As discussed in the introduction, we now group the wave packets into bushes and a (nearly) non-overlapping collection (see Figure \ref{figure:multi_packet}). This argument is inspired by Bourgain's bush argument from \cite{Bourgain91}, and we also refer the reader to \cite[Proposition 2.2]{Wolff99}. \\

Before we state main proposition, we define the truncated and fattened \( \ell^\infty\)-cone 
\begin{equation}\label{wp:eq_fattened_cone}
\KNtilde := \{ (t,x) \in [t_0,t_0+N] \times \rfour \colon \|x-x_0\|_\infty \leq 16N-|t-t_0| \} ~. 
\end{equation}
The significance of \( \KNtilde \) will be explained in Section \ref{section:flux} and Section \ref{section:a_priori}. For now, we encourage the reader to treat \( \KNtilde \) as space-time cube of scale \( N \). 

\begin{prop}[Wave packet decomposition and bushes]\label{wp:prop_wp_bushes}
Let \( \{ f_k^\pm \}_{k} \subseteq L_x^2(\rd) \) be a family of functions, where \( \| k \|_\infty \in (N/2,N] \), and \( \supp \widehat{f^\pm_k}\subseteq k + [-1,1]^d \). Let \( t_0 \in N\mathbb{N}_0 \), let \( x_0 \in N\zd \), and let the wave packets \( \{ f_{k,l;t_0}^\pm \} \) be as in \mbox{Corollary \ref{wp:cor_wave_packet_decomposition}}. Furthermore, let \( \QN \) be a collection of disjoint space-time cubes with sidelength \( \sim N^{\delta} \) covering \( \KNtilde \). 
We group the wave packets according to their amplitude by setting 
\begin{equation}\label{wp:eq_am}
\Am = \Amall := \{ (k,l) \in \zd \times \zd \colon \| f_{k,l;t_0}^\pm \|_{L_x^2(\rd)} \in [2^m,2^{m+1}], \| l - x_0 \|_\infty \leq 3N\} ~. 
\end{equation}
Then, there exists a family of bushes \( \{ \Bj \}_{j} = \{ \Bjall \}_{j} \), where \( j=1,\hdots, J^{N,\pm}_{m,t_0,x_0} \), and a nearly 
non-overlapping set \( \Dm = \Dmall \), depending only on the set \( \Am \), so that the following holds:
\begin{enumerate}
\item \label{wp:item_partition} The sets form a partition of \( \Am \), i.e.,  \begin{equation}
\Am = \Dm ~ \overset{\bullet}{\bigcup} ~ \left( \mathsmaller{\overset{\bullet}{\bigcup} }_{j=1,\hdots, J} \Bj \right)
\end{equation}
\item \label{wp:item_number_wp} We have the bound on the number of wave packets \begin{equation}\label{wp:eq_number_wp}
\sum_{x_0\in N\zd} \sum_{m\in \mathbb{Z}} 2^{2m} \# \Amall \lesssim \sum_{k} \| f_k\|_{L_x^2(\rd)}^2~. 
\end{equation}
\item \label{wp:item_bush_elements} Each bush \( \Bj\) contains at least \( \mu=\mu(N,m) :=N^{-\frac{1}{2}} \# \Am \) wave packets. 
\item  \label{wp:item_bush_intersection}For each bush \( \Bj \), all corresponding wave packets intersect in the same region of space-time. More precisely, there exists a cube \( Q\in \QN \) s.t. 
\begin{equation}\label{wp:eq_bush_intersection}
T_{k,l;t_0}^\pm ~\mathsmaller{\bigcap}~ 2Q \neq \emptyset \qquad \forall (k,l) \in \Bj ~. 
\end{equation}
\item \label{wp:item_non_overlapping} At \( \mu= N^{-\frac{1}{2}}\# \Am  \) wave packets in \( \Dm \) overlap, i.e.,  we have for all cubes \( Q \in \QN \) that 
\begin{equation}
\# \{ (k,l) \in \Dm \colon T_{k,l;t_0}^\pm~ \mathsmaller{\bigcap}~ 2Q \neq \emptyset \} < P~. 
\end{equation} 
\end{enumerate}
\end{prop}

The choice of the number of packets/multiplicity \( \mu=N^{-\frac{1}{2}} \# \Am \) will be justified in the proof of Proposition \ref{apriori:prop_single_scale}, see \eqref{apriori:eq_choice_P}. The parameter \( \mu \) corresponds to the multiplicity parameter in Bourgain's bush argument, see \cite[Proposition 2.2]{Wolff99}. 

\begin{rem}\label{wp:rem_deterministic_collections}
We will later apply this proposition to a set of random functions \( \{ \epsilon_k f_k^\pm \}_k \). From \eqref{wp:eq_am}, it follows that the sets \( \Am \), and hence also \( \Bj \) and \( \Dm\), do not depend on the random signs \( \{ \epsilon_k \} \). 
\end{rem}

\begin{proof}
Let us first prove the inequality in \ref{wp:item_number_wp}. From Corollary \ref{wp:cor_wave_packet_decomposition}, it follows that
\begin{equation*}
\sum_{x_0\in N\zd} \sum_{m\in \mathbb{Z}} 2^{2m} \# \Amall \lesssim \sum_{\substack{ k,l\in \zd \\ \| k\|_\infty \in (N/2,N]}} \| f_{k,l;t_0}^\pm \|_{L_x^2}^2 \lesssim \sum_{\substack{k\in \zd\\ \| k\|_\infty \in(N/2,N]}} \| f_k^\pm \|_{L_x^2}^2~. 
\end{equation*}

We now construct the sets \( \Bjall \) and \( \Dmall \). 
To simplify the expressions, we drop the super- and subscripts \( \pm, N, t_0\), and \( x_0\) from our notation. The basic idea is to form the bushes through a greedy selection algorithm. For any \( Q \in \QN \), we define
\begin{equation}\label{wp:eq_tq}
\TQ := \{ (k,l)\in \Am \colon T_{k,l} ~ \mathsmaller{\bigcap} ~ Q \neq \emptyset \}~. 
\end{equation}
We further set \( \mathscr{T}_m^{(1)}(Q) := \TQ \). We then choose a cube \( Q_1 \in \QN \) such that 
\begin{equation*}
\# \mathscr{T}_m^{(1)}(Q_1) = \max_{Q\in \QN} \# \mathscr{T}_m^{(1)}(Q)~,
\end{equation*}
and define the first bush as \( \mathscr{B}_{1,m} :=  \mathscr{T}_m^{(1)}(Q_1) \). By setting 
\begin{equation*}
\mathscr{T}_m^{(2)}(Q) = \mathscr{T}_m^{(1)}(Q)\backslash \mathscr{B}_{1,m}~,
\end{equation*}
we remove all of the wave packets in the first bush from the collection. We then iteratively define \( \Bj := \mathscr{T}^{(j)}_m(Q_j) \), where 
\begin{equation*}
\# \mathscr{T}_m^{(j)}(Q_j) = \max_{Q\in \QN} \# \mathscr{T}_m^{(j)}(Q)~,
\end{equation*}
and the collections \(  \mathscr{T}_m^{(j)}(Q) \) are defined as \(  \mathscr{T}_m^{(j-1)}(Q)\backslash \Bjminus \). Once \(  \mathscr{T}_m^{(j+1)}(Q_{j+1}) < \mu \), we no longer create a new bush, and instead stop the algorithm. Since \( \Am \) contains at most \( \sim N^{8} \) wave packets, the greedy selection algorithm terminates after finitely many steps. From the construction, we see that the sets \( \Bj \subseteq \zd \times \zd \)  are disjoint (even though the corresponding tubes may still overlap). Finally, we define the collection \( \Dm \) by \begin{equation*}
\Dm := \Am \backslash \mathsmaller{\bigcup}_{j=1}^J \Bj~. 
\end{equation*}
The properties \ref{wp:item_partition}, \ref{wp:item_bush_elements}, \ref{wp:item_bush_intersection}, and \ref{wp:item_non_overlapping} then follow directly from the construction. 
\end{proof}

We now prove a probabilistic estimate for the wave packets with random coefficients. 

\begin{prop}[Square-root cancellation for wave packets]\label{wp:prop_square_root_cancellation}
Let \( \{ f_k^\pm \}_{k} \subseteq L_x^2(\rd) \) be a deterministic family of functions, where \( k \in \zd \) satisfies \( \| k \|_\infty \in (N/2,N] \), and assume that \( \sum_{k} \| f_k^\pm \|_{L_x^2(\rd)}^2 \lesssim 1 \). Let \( \theta > 0 \) be a parameter, and let \( C_d > 0 \) be any large absolute constant. Then, we have for all \( m \in \mathbb{Z} \) satisfying \( - C_d \log(N) \leq m \leq C_d \log(N) \) that 
\begin{equation}\label{wp:eq_square_root_bushes}
\mathbb{E} \Bigg[ \sup_{ \substack{ t_0=0,\hdots \lfloor N^\theta \rfloor N \\ x_0\in N\zd}} ~\sup_{j=1,\hdots, J^{N,\pm}_{m,t_n,x_0}} 
\frac{\Big\| \sum_{(k,l)\in \Bjall} \epsilon_k \exp(\pm i (t-t_0)|\nabla|) f_{k,l;t_0}^\pm \Big\|_{L_{t,x}^\infty(\reals \times \rd)}}{2^m (\# \Bjall)^{\frac{1}{2}} }\Bigg] 
\lesssim N^{\delta d} 
\end{equation}
and
\begin{equation}\label{wp:eq_square_root_disjoint}
\mathbb{E} \Bigg[ \sup_{ \substack{t_0=0,\hdots \lfloor N^\theta \rfloor N \\ x_0\in N\zd}} 
\frac{\Big\| \sum_{(k,l)\in \Dmall} \epsilon_k \exp(\pm i (t-t_0)|\nabla|) f_{k,l;t_0}^\pm \Big\|_{L_{t}^\infty L_x^\infty([t_0,t_0+N]  \times \rd)}}{2^m \mu^{\frac{1}{2}} }\Bigg] 
\lesssim N^{\delta d} 
~.
\end{equation}
Here, \( \mu= N^{-\frac{1}{2}} \# \Amall \) is as in Proposition \ref{wp:prop_wp_bushes}. To be perfectly precise, we use the convention \( 0/0:= 0 \) in \eqref{wp:eq_square_root_disjoint}.
\end{prop}
The expressions in \eqref{wp:eq_square_root_bushes} and \eqref{wp:eq_square_root_disjoint} may seem complicated. To make sense of them, recall that the square function heuristic predicts that \( \sum_k \epsilon_k a_k \) is roughly of size \( \sim (\sum_k a_k^2 )^{\frac{1}{2}} \). Then, Proposition \ref{wp:prop_square_root_cancellation} simply states that the square function heuristic can be justified for all relevant amplitudes, for all relevant times, all positions, all families of bushes, and all non-overlapping collections.  \\
For instance, let us heuristically motivate \eqref{wp:eq_square_root_disjoint}. By the definition of \( \Dmall \), any fixed point in the space-time region \( [t_0,t_0+N] \times \rd \) is contained in the (moral) support of at most \( \mu \) wave packets. Since each of the wave packets has amplitude \( \sim 2^m \), and they all correspond to different frequencies \( k \in \zd \), the square-function heuristic predicts a contribution of size \( \sim 2^m \mu^{\frac{1}{2}} \). 
\begin{proof}
In this proof, we make extensive use of Lemma \ref{prelim:lem_suprema_subgaussian}. First, we prove that the suprema in \eqref{wp:eq_square_root_bushes} and \eqref{wp:eq_square_root_disjoint} are over at most \( N^{\mathcal{O}(C_d)} \)-many terms. From \eqref{wp:eq_number_wp}, it follows for all \( m\geq -C_d \log(N) \) that 
\begin{align*}
\sum_{x_0\in N\zd} \# \Amall \lesssim 2^{-2m} \sum_{k,l} \| f_{k,l}^\pm \|_{\lxtwo}^2 
\lesssim 2^{-2m} \sum_{k} \| f_k^\pm \|_{\lxtwo}^2 
\lesssim N^{2C_d}~. 
\end{align*}
Thus, this bounds the number of all wave packets with amplitude \( \sim 2^m \). 
Since each bush \( \Bjall \) contains at least one wave packet, the supremum in \eqref{wp:eq_square_root_bushes} is over at most \( \mathcal{O}(N^{2C_d}) \)  non-zero terms. The same applies to the non-overlapping families \( \Dmall \) in \eqref{wp:eq_square_root_disjoint}. From Lemma \ref{prelim:lem_suprema_subgaussian}, it then suffices to obtain uniform sub-gaussian bounds on each individual term in  \eqref{wp:eq_square_root_bushes} and \eqref{wp:eq_square_root_disjoint}. \\
We start with the contribution of the bushes. To simplify the notation, we write \( \Bj = \Bjall \). From Bernstein's inequality and Lemma \ref{prelim:lem_inf_removal}, we have for all \( 2 \leq p < \infty \) that 
\begin{align*}
&\Big \| \sum_{(k,l)\in \Bj} \epsilon_k \exp(\pm i(t-t_0)|\nabla|) f_{k,l;t_0}^+ \Big\|_{L_t^\infty L_x^\infty(\reals\times \rd)} \\
&\leq N^{\frac{d+1}{p}} \Big \| \sum_{(k,l)\in \Bj} \epsilon_k \exp(\pm i(t-t_0)|\nabla|) f_{k,l;t_0}^+ \Big\|_{L_t^p L_x^p(\reals\times \rd)} ~.
\end{align*}
Before we utilize the randomness, we observe that for each \( k \in \zd \) at most \( \mathcal{O}(N^{\delta d}) \) tubes \( T_{k,l;t_0}^{\pm} \) can intersect a space-time cube of sidelength \(\sim N^{\delta} \). As a result, it follows from \eqref{wp:eq_bush_intersection} that 
\begin{equation*}
\# \{ l\in \zd \colon (k,l)\in \Bj \} \lesssim N^{\delta d} ~. 
\end{equation*} 
For all \( p\leq r < \infty \), we then obtain from Minkowski's integral inequality, Khintchine's inequality, and the refined Strichartz estimate (Corollary \ref{prelim:cor_refined_strichartz}) that
\begin{align*}
&\Big \| \sum_{(k,l)\in \Bj} \epsilon_k \exp(\pm i(t-t_0)|\nabla|) f_{k,l;t_0}^+ \Big\|_{L_\omega^rL_t^\infty L_x^\infty(\Omega\times\reals\times \rd)} \\
&\lesssim N^{\frac{d+1}{p}} \Big \| \sum_{(k,l)\in \Bj} \epsilon_k \exp(\pm i(t-t_0)|\nabla|) f_{k,l;t_0}^+ \Big\|_{L_t^p L_x^pL_\omega^r(\reals\times \rd\times \Omega)} \\
&\lesssim \sqrt{r} N^{\frac{d+1}{p}}  \Big \|  \bigg( \sum_{k\in \zd} \Big(\sum_{l\colon (k,l)\in \Bj}  \exp(\pm i(t-t_0)|\nabla|) f_{k,l;t_0}^+ \Big)^2 \bigg)^{\frac{1}{2}}\Big\|_{L_t^p L_x^p(\reals\times \rd)} \\
&\lesssim  \sqrt{r} N^{\frac{d+1}{p}+ \frac{\delta d}{2}}\Big \| \exp(\pm i(t-t_0)|\nabla|) f_{k,l;t_0}^+  \Big\|_{L_t^p L_x^p \ell_{k,l}^2(\reals\times \rd \times \Bj)} \\
&\lesssim  \sqrt{r} N^{\frac{d+1}{p}+ \frac{\delta d}{2}}\Big \| \exp(\pm i(t-t_0)|\nabla|) f_{k,l;t_0}^+  \Big\|_{ \ell_{k,l}^2 L_t^p L_x^p (\Bj\times \reals\times \rd)} \\
&\lesssim \sqrt{r} N^{\frac{d+2}{p}+ \frac{\delta d}{2}} \Big \| f_{k,l;t_0}^+ \Big\|_{\ell_{k,l}^2 L_x^2(\Bj\times \rd)} \\
&\lesssim \sqrt{r} N^{\frac{d+2}{p} + \frac{\delta d}{2}} 2^m (\# \Bj)^{\frac{1}{2}}~. 
\end{align*}
By taking \( p\geq 2 \) to be sufficiently large, we then obtain the desired sub-gaussian bound. This completes the proof of \eqref{wp:eq_square_root_bushes}. \\
We now control the contribution of a single non-overlapping family \( \Dm = \Dmall \). For the technical aspects of this part, recall that the collection \( \QN \) from Proposition \ref{wp:prop_wp_bushes} covers \( \KNtilde \), but due the definition of \( \Am \) in \eqref{wp:eq_am}, all the tubes \( T_{k,l;t_0}^\pm \) with indices in \( \Dm \) are contained in the region \( \| x- x_0 \|_\infty \leq 6N \). This gives us sufficient room for the following argument. \\
We let \(  2 \leq p < \infty \). As before, it follows from Bernstein's inequality and Lemma \ref{prelim:lem_inf_removal} that 
\begin{align*}
&\Big\| \sum_{(k,l)\in \Dmall} \epsilon_k \exp(\pm i (t-t_0)|\nabla|) f_{k,l;t_0}^\pm \Big\|_{L_t^\infty L_x^\infty([t_0,t_0+N] \times \rd)} \\
&\lesssim N^{\frac{d+1}{p}} \Big\| \sum_{(k,l)\in \Dmall} \epsilon_k \exp(\pm i (t-t_0)|\nabla|) f_{k,l;t_0}^\pm \Big\|_{L_t^p L_x^p([t_0,t_0+N]  \times \rd)}
\end{align*}
For all \( p \leq r <\infty\), we then obtain from Minkowski's integral inequality and Khintchine's inequality that
\begin{align*}
&\Big\| \sum_{(k,l)\in \Dm} \epsilon_k \exp(\pm i (t-t_0)|\nabla|) f_{k,l;t_0}^\pm \Big\|_{L_\omega^r L_t^\infty L_x^\infty(\Omega \times [t_0,t_0+N]  \times \rd)} \displaybreak[1] \\
&\lesssim N^{\frac{d+1}{p}} \Big\| \sum_{(k,l)\in \Dm} \epsilon_k \exp(\pm i (t-t_0)|\nabla|) f_{k,l;t_0}^\pm \Big\|_{L_t^p L_x^p L_\omega^r([t_0,t_0+N]  \times \rd \times \Omega)} \displaybreak[1] \\
&\lesssim \sqrt{r} N^{\frac{d+1}{p}} \Big\|  \bigg( \sum_{k\in \zd} \big( \sum_{l\colon (k,l)\in \Dm} \exp(\pm i (t-t_0)|\nabla|) f_{k,l;t_0}^\pm  \big)^2 \bigg)^{\frac{1}{2}} \Big\|_{L_t^p L_x^p([t_0,t_0+N]  \times \rd)} \displaybreak[1] \\
&\lesssim \sqrt{r} N^{\frac{d+1}{p}} \Big\|  \bigg( \sum_{k\in \zd} \big( \sum_{l\colon (k,l)\in \Dm} \exp(\pm i (t-t_0)|\nabla|) f_{k,l;t_0}^\pm  \big)^2 \bigg)^{\frac{1}{2}} \Big\|_{L_t^p L_x^p(\KNtilde)}  \\
&\quad + \sqrt{r} N^{\frac{d+1}{p}} \Big\|  \bigg( \sum_{k\in \zd} \big( \sum_{l\colon (k,l)\in \Dm} \exp(\pm i (t-t_0)|\nabla|) f_{k,l;t_0}^\pm  \big)^2 \bigg)^{\frac{1}{2}} \Big\|_{L_t^p L_x^p(([t_0,t_0+N]  \times \rd)\backslash \KNtilde)}
\end{align*}
Since \( \mu = N^{-\frac{1}{2}} \# \Am \geq N^{-\frac{1}{2}} \), the bound on \( ([t_0,t_0+N]  \times \rd)\backslash \KNtilde \) easily follows from the decay estimate \eqref{wp:eq_decay}.  Thus, we now control the contribution on \( \KNtilde \). From Hölder's inequality, we have that 
\begin{align*}
 &\Big\|  \bigg( \sum_{k\in \zd} \big( \sum_{l\colon (k,l)\in \Dm}  \exp(\pm i (t-t_0)|\nabla|) f_{k,l;t_0}^\pm  \big)^2 \bigg)^{\frac{1}{2}} \Big\|_{L_t^p L_x^p(\KNtilde)} \\
 &\lesssim N^{\frac{d+1}{p}}  \Big\|  \bigg( \sum_{k\in \zd} \big( \sum_{l\colon (k,l)\in \Dm} \exp(\pm i (t-t_0)|\nabla|) f_{k,l;t_0}^\pm  \big)^2 \bigg)^{\frac{1}{2}} \Big\|_{L_t^\infty L_x^\infty(\KNtilde)}
\end{align*}
Now pick any cube \( Q \in \QN \). In analogy to \eqref{wp:eq_tq}, we define the collection of ``remaining'' tubes by 
\begin{equation*}
\mathscr{T}_{m}^{r}(Q):= \{ (k,l)\in \Dm\colon T_{k,l;t_0}^{\pm} ~ \mathsmaller{\bigcap} ~ 2Q \neq \emptyset\}~. 
\end{equation*}
From Proposition \ref{wp:prop_wp_bushes}, it follows that \( \# \mathscr{T}_{m}^{r}(Q)  \leq \mu \). As above, we have for each frequency \( k \in \zd \) the bound \( \# \{ l \in\zd \colon (k,l) \in \mathscr{T}_{m}^{r}(Q) \} \lesssim N^{\delta d} \). Using the decay estimate \eqref{wp:eq_tq} to treat distant wave packets, we obtain
\begin{align*}\displaybreak[1]
&\Big\|  \bigg( \sum_{k\in \zd} \big( \sum_{l\colon (k,l)\in \Dm} \exp(\pm i (t-t_0)|\nabla|) f_{k,l;t_0}^\pm  \big)^2 \bigg)^{\frac{1}{2}} \Big\|_{L_t^\infty L_x^\infty(Q)} \\
&\lesssim \Big\|  \bigg( \sum_{k\in \zd} \big( \sum_{l\colon (k,l)\in  \mathscr{T}_{m}^{r}(Q) }  \exp(\pm i (t-t_0)|\nabla|) f_{k,l;t_0}^\pm  \big)^2 \bigg)^{\frac{1}{2}} \Big\|_{L_t^\infty L_x^\infty(Q)}\displaybreak[1] \\
&\quad +\Big\|  \bigg( \sum_{k\in \zd} \big( \sum_{l\colon (k,l)\in \Dm \backslash \mathscr{T}_{m}^{r}(Q) }  \exp(\pm i (t-t_0)|\nabla|) f_{k,l;t_0}^\pm  \big)^2 \bigg)^{\frac{1}{2}} \Big\|_{L_t^\infty L_x^\infty(Q)}\displaybreak[1] \\
&\lesssim N^{\frac{\delta d}{2}} \Big\|  \exp(\pm i (t-t_0)|\nabla|) f_{k,l;t_0}^\pm  \Big\|_{L_t^\infty L_x^\infty \ell_{k,l}^2(Q\times \mathscr{T}_{m}^{r}(Q))} + N^{-10C_d} \displaybreak[1]\\
&\lesssim N^{\frac{\delta d}{2}} 2^m (\# \mathscr{T}_{m}^{r}(Q))^{\frac{1}{2}} + N^{-10C_d} \displaybreak[1]\\
&\lesssim N^{\frac{\delta d}{2}} 2^m \mu^{\frac{1}{2}}~. 
\end{align*}
After taking the supremum over all cubes \( Q \in \QN \), we finally arrive at
\begin{align*}
&\Big\| \sum_{(k,l)\in \Dm} \epsilon_k \exp(\pm i (t-t_0)|\nabla|) f_{k,l;t_0}^\pm \Big\|_{L_\omega^r L_t^\infty L_x^\infty(\Omega \times [t_0,t_0+N]  \times \rd)} 
\lesssim  \sqrt{r} N^{2 \frac{d+2}{p}+ \frac{\delta d}{2}}  2^m \mu^{\frac{1}{2}}~. 
\end{align*} 
By choosing \( p \geq 2 \) sufficiently large, we arrive at the desired sub-gaussian bound.
\end{proof}

\begin{definition}[Wave packet ``norm''] \label{wp:def_wp_norm}
Let \( (f_0,f_1)\in H^s(\rd) \times H^{s-1}(\rd) \) and let \( f_k^\pm \) be defined as in \eqref{wp:eq_microlocal_to_wiener} and \eqref{wp:eq_half_wave_decomposition}.
For any \( N_0 \geq 1 \), we then define the wave packet ``norm'' of the random data \( ( f_0^\omega,f_1^\omega) \) by
\begin{align*}
&\| (f_0^\omega,f_1^\omega) \|_{\WP(N_0)} := \\
&\sum_{ \substack{N\geq N_0\\ |m|\leq C_d \log(N)}} N^{-2\delta d}\sup_{ \substack{  t_0=0,\hdots \lfloor N^\theta \rfloor N \\ x_0\in N\zd}} \sup_{j=1,\hdots, J^{N,\pm}_{m,t_0,x_0}} 
\frac{\Big\| \sum_{(k,l)\in \Bjall} \epsilon_k \exp(\pm i (t-t_0)|\nabla|) f_{k,l;t_0}^\pm \Big\|_{L_t^\infty L_x^\infty(\reals \times \rd)}}{2^m (\# \Bjall)^{\frac{1}{2}} } \\
&~+ \sum_{ \substack{N\geq N_0\\ |m|\leq C_d \log(N)}} N^{-2\delta d} \sup_{ \substack{  t_0=0,\hdots \lfloor N^\theta \rfloor N \\ x_0\in N\zd}} 
\frac{\Big\| \sum_{(k,l)\in \Dmall} \epsilon_k \exp(\pm i (t-t_0)|\nabla|) f_{k,l;t_0}^\pm \Big\|_{L_{t}^\infty L_x^\infty([t_0,t_0+N]  \times \rd)}}{2^m \mu^{\frac{1}{2}} }
\end{align*}
While the quantity \( \| (f_0^\omega,f_1^\omega) \|_{\WP(N_0)}  \) measures the size of the wave packets (over their expected size), it is certainly far from being an actual norm. 
\end{definition}

\begin{cor}\label{wp:cor_wp_norm}
Let \( (f_0,f_1)\in H^s(\rd) \times H^{s-1}(\rd) \) and let \( f_k^\pm \) be defined as in \eqref{wp:eq_microlocal_to_wiener} and \eqref{wp:eq_half_wave_decomposition}. Furthermore, we assume that \( N_0 = N_0( \{ Y_{k,l} \}) \) satisfies
\begin{equation*}
\sum_{\|k \|_\infty \geq N_0/2} \left( \| f_{0;k} \|_{L_x^2(\rd)}^2 + \| f_{1;k} \|_{\dot{H}_x^{-1}(\rd)}^2  \right)\lesssim 1 ~. 
\end{equation*}
Then, it holds that 
\begin{equation*}
\mathbb{E}_\epsilon\| (f_0^\omega,f_1^\omega) \|_{\WP(N_0)} < \infty~,
\end{equation*}
where \( \mathbb{E}_\epsilon \) denotes the expectation over the random signs \( \{ \epsilon_k \} \). 
\end{cor}

\begin{proof}
This follows directly from Proposition \ref{wp:prop_square_root_cancellation} and Definition \ref{wp:def_wp_norm}. 
\end{proof}

For the bootstrap argument in Section \ref{section:a_priori}, it will be convenient to create a small forcing term by truncating to high frequencies. If \( N_{\hi} = N_{\hi}(\omega) \) is a (possibly random) frequency parameter, we set 
\begin{equation}\label{wp:eq_high_frequency}
F_{\hi} := \sum_{N\geq N_{\hi}} \left( \cos(t|\nabla|) f_{0,N}^\omega + \frac{\sin(t|\nabla|)}{|\nabla|} f_{1,N}^\omega\right)~. 
\end{equation}

\begin{prop}[Truncation to high frequencies]\label{wp:prop_smallness}
Let \( 0 <\eta \leq 1\) be an absolute constant and let \( s > \frac{1}{3} \). Let \( (f_0,f_1) \in H_x^s(\rfour)\times H_x^{s-1}(\rfour) \), let \( (f_0^\omega,f_1^\omega) \) be their microlocal randomizations, and let \( \{ f_k^\pm \} \) be as in \eqref{wp:eq_microlocal_to_wiener} and \eqref{wp:eq_half_wave_decomposition}. Then, there exists a random frequency parameter \( N_{\hi} \geq \eta^{-1} \) such that 
\begin{align*}
\| (P_{\geq N_{\hi}/4} f_0^\omega ,  P_{\geq N_{\hi}/4} f_1^\omega) \|_{H^s \times H^{s-1}} &\leq \eta~, \\
\| (f_0^\omega,f_1^\omega) \|_{Z(N_{\hi})} &\leq \eta~, \\
\|(f_0^\omega,f_1^\omega) \|_{\WP(N_{\hi})} &\leq \eta~, \\
\| F_{\hi} \|_{L_t^3L_x^6(\reals\times \rfour)}&\leq \eta~.
\end{align*}
\end{prop}

\begin{proof}
We only need to combine the previous estimates. 
From Lemma \ref{prelim:lem_random_hs}, it follows that 
\begin{equation*}
\sum_{N\geq 2} \sum_{\|k\|_\infty \in (N/2,N]} \left( N^{2s} \| f_{0;k} \|_{L_x^2}^2 + N^{2(s-1)} \| f_{1;k} \|_{L_x^2}^2 \right) < \infty \qquad \text{a.s.}
\end{equation*}
From dominated convergence, it then follows that there exists some random frequency \( N_{\hi,1} \), depending only on the random variables \( \{ Y_{k,l} \} \),  satisfying
\begin{equation*}
\sum_{N\geq N_{\hi,1}} \sum_{\|k\|_\infty \in (N/2,N]} \left( N^{2s} \| f_{0;k} \|_{L_x^2}^2 + N^{2(s-1)} \| f_{1;k} \|_{L_x^2}^2 \right) \leq \eta \qquad \text{a.s.}
\end{equation*}
From Corollary \ref{wp:cor_wp_norm}, it then follows that 
\begin{equation*}
\mathbb{E}_\epsilon\| (f_0^\omega,f_1^\omega) \|_{\WP(N_{\hi,1})} < \infty~.
\end{equation*}
From dominated convergence, it follows that there exists a random frequency \( N_{\hi,2} \), depending on the random variables \( \{ \epsilon_k \} \) and \( \{ Y_{k,l} \} \), which satisfies
\begin{equation*}
\| (f_0^\omega,f_1^\omega) \|_{\WP(N_{\hi,2})}\leq \eta \qquad \text{a.s.}
\end{equation*}
By similar arguments, it also follows from  Proposition \ref{prob:lem_probabilistic_strichartz} and Corollary \ref{prob:cor_long_time_decay} that there exists random frequencies \( N_{\hi,3} \) and \( N_{\hi,4} \) such that 
\begin{equation*}
\| (f_0^\omega,f_1^\omega) \|_{Z(N_{\hi,3})} \leq \eta \quad \text{and} \quad \sum_{N\geq N_{\hi,4}}\Big \|\cos(t|\nabla|) f_{0,N}^\omega + \frac{\sin(t|\nabla|)}{|\nabla|} f_{1,N}^\omega \Big \|_{L_t^3L_x^6(\reals\times \rfour)}\leq \eta
\end{equation*}
For the second inequality, we have used the condition \( s > \frac{1}{3} \).\\
By choosing \( N_{\hi} := \max( 4 N_{\hi,1},N_{\hi,2} ,N_{\hi,3} ,N_{\hi,4},\eta^{-1})  \), we arrive at the desired conclusion. 
\end{proof}

\section{Nonlinear evolution: Local well-posedness, stability theory, and flux estimates}
In this section, we first apply to Da Prato-Debussche trick \cite{PD02} to the nonlinear wave equation with random initial data. Then, we recall certain properties of the (forced) energy critical nonlinear wave equation. In our exposition of the local well-posedness and stability theory, we mainly rely on \cite{DLM17}. The flux estimate already played a major role in the author's work on almost sure scattering for the radial energy critical NLW \cite{Bringmann18}, and we loosely follow parts of \cite[Section 6]{Bringmann18}. \\

Let \( N_{\hi} \) be as in Proposition \ref{wp:prop_smallness}, and let \( F = F_{\hi} \) be as in \eqref{wp:eq_high_frequency}. We then decompose the solution \( u \) of \eqref{intro:eq_nlw_random} by setting \( v:= u - F \). Then, the nonlinear component \( v \) solves the forced nonlinear wave equation
\begin{equation}\label{prelim:eq_forced_nlw}
\begin{cases}
-\partial_{tt} v+ \Delta v = (v+F)^3 \qquad (t,x)\in \reals\times \reals^4~,\\
v|_{t=0}  = v_0 \in \dot{H}^1, \quad \partial_t v|_{t=0} = v_1 \in L^2~,
\end{cases}
\end{equation}
where \( v_0 := u_0 + f^\omega_{0,<N_{\hi}} \) and \(v_1 := u_1 + f^\omega_{1,<N_{\hi}} \). The randomness in the initial data \( (v_0,v_1) \) is not important, and we treat it as arbitrary data in the energy space. For the rest of this section, we treat \( F \) as an arbitrary forcing term in \( L_t^3 L_x^6(\reals\times \rfour) \), since the finer properties of \( F \) will only be relevant in Section \ref{section:a_priori}.

\subsection{Local well-posedness and stability theory}

In this section, we recall the local well-posedness of \eqref{prelim:eq_forced_nlw}. Using stability theory, we recall the reduction of Theorem \ref{main_theorem} to an a-priori energy bound. These results are well-known in the literature, see e.g. \cite{DLM17,Pocovnicu17}. 

\begin{lem}[{Local well-posedness \cite[Lemma 3.1]{DLM17}}]\label{prelim:lem_lwp}
Let \( (v_0,v_1) \in \dot{H}^1(\rfour)\times L^2(\rfour) \) and \( F \in L_t^3 L_x^6([0,\infty)\times \rfour) \). Then, there exists a time \( T > 0 \) and a unique solution \( v \colon [0,T) \times \rfour \rightarrow \reals \) satisfying
\begin{equation*}
v \in C_t^0 \dot{H}^1_x([0,T)\times \rfour) ~ \mathsmaller{\bigcap} ~ L_t^3 L_x^6([0,T)\times \rfour) \qquad \text{and} \qquad \partial_t v \in C_t^0 L_x^2([0,T)\times \rfour)~. 
\end{equation*}
\end{lem}

Using stability theory, \cite{DLM17} proved the following proposition. 

\begin{prop}[{Reduction to an a-priori energy bound \cite[Theorem 1.3]{DLM17}}]\label{prelim:prop_reduction}
Let \( (v_0,v_1) \in \dot{H}^1(\rfour)\times L^2(\rfour) \) and \( F \in L_t^3 L_x^6(\reals\times \rfour) \). Let \( v \) be a solution of \eqref{prelim:eq_forced_nlw}, and let \( T_+ > 0 \) be its maximal time of existence. Furthermore, assume the a-priori energy bound \begin{equation*}
\sup_{t\in [0,T_+)} E[v](t) < \infty~.
\end{equation*}
Then, \( v \) is a global solution and satisfies the global space-time bound \( \| v\|_{L_t^3 L_x^6([0,\infty)\times \rfour)} < \infty \). As a result, there exist a scattering states \( (v_0^+,v_1^+)\in \dot{H}^1(\rfour)\times L^2(\rfour) \) such that 
\begin{equation*}
\lim_{t\rightarrow +\infty} \| (v(t)-W(t)(v_0^+,v_1^+), \partial_t v(t) - \partial_t W(t) (v_0^+,v_1^+)) \|_{\dot{H}^1\times L^2} = 0~. 
\end{equation*}
\end{prop}

Using Lemma \ref{prelim:lem_lwp} and Proposition \ref{prelim:prop_reduction}, we have reduced the proof of Theorem \ref{main_theorem} to an a-priori bound on the energy of \( v \).

\subsection{Flux estimates}\label{section:flux}
As before, we let \( v \colon I \times \rfour \rightarrow \reals \) be a solution to the forced equation \eqref{prelim:eq_forced_nlw}. 
Recall that the (symmetric) energy-momentum tensor for the energy critical nonlinear wave equation is given by 
\begin{align*}
T^{00}= T^{00}[v] &:= \frac{1}{2} \big( (\partial_t v)^2+ |\nabla v|^2\big) + \frac{1}{4} v^4~, \\
T^{j0} = T^{j0}[v] &:= - \partial_t v \cdot \partial_{x_j} v~, \\
T^{jk} = T^{jk}[v] &:= \partial_{x_j} v \partial_{x_k} v - \frac{\delta_{jk}}{4} (-\partial_{tt}+ \Delta)(v^2) + \frac{\delta_{jk}}{2} v^4~. 
\end{align*}
The component \( T^{00} \) is the energy density, the component \( T^{j0} \) is the \( j\)-th momentum/energy flux, and the components \( T^{jk} \) are called the momentum flux. If \( v \) solves the energy critical nonlinear wave equation \eqref{intro:eq_nlw}, then the energy-momentum tensor is divergence free. This fails for solutions to the forced equation \eqref{prelim:eq_forced_nlw}; however, one can still expect that the error terms have lower order. Setting \( \mathcal{N} = (v+F)^3 - v^3 \), a short computation shows that 
\begin{align}
\partial_t T^{00} + \partial_{x_k} T^{0k}& = - \mathcal{N} \partial_t v  \label{prelim:eq_energy_flux}~,  \\
\partial_t T^{j0} + \partial_{x_k} T^{jk} &= \mathcal{N} \partial_{x_j} v - \frac{1}{2} \partial_{x_j} ( \mathcal{N} v ) \label{prelim:eq_momentum_flux}~. 
\end{align}
As in earlier work on almost sure scattering for radial data \cite{Bringmann18,DLM17,DLM18,KMV17}, the main goal of this paper is to bound the energy of \( v \). In terms of  the energy-momentum tensor, the (total) energy can be written as 
\begin{equation*}
E[v](t):= \int_{\rfour} T^{00}(t,x) \dx ~. 
\end{equation*}
For future use, we record the following consequence of \eqref{prelim:eq_energy_flux}. 
\begin{lem}[Total energy increment]
Let \( v \colon I \times \rfour \rightarrow \reals \) be a solution of \eqref{prelim:eq_forced_nlw}, and let \( a,b \in I \) with \( a \leq b\). Then, we have that 
\begin{equation}
E(b)-E(a)  = - \int_a^b \int_{\rfour} \mathcal{N} \partial_t v \dx \dt 
			\leq 6 \int_a^b \int_{\rfour} |F| |v|^2 |\partial_t v|\dx \dt + 3 \int_a^b \int_{\rfour} |F|^3 |\partial_t v|\dx \dt~. \label{prelim:eq_total_energy_increment}
\end{equation}
\end{lem}
We will later see that the second summand on the right-hand side of \eqref{prelim:eq_total_energy_increment} can be bounded directly using Hölder's inequality and probabilistic Strichartz estimates. In contrast, no such estimate is available for the first summand, and we need the wave packet decomposition to control this term. 

Once we employ the wave packet decomposition, it will be natural to study the energy  on a time and length scale \( \sim N \geq 1\). We fix \( t_0 \in N \mathbb{N}_0 \) and \( x_0 \in N\zfour\), and define the local energy
\begin{equation}
E_{t_0,x_0}^N[v](t) := \int_{\|x-x_0 \|_\infty \leq 2N-|t-t_0|} T^{00}(t,x)\dx, \qquad \text{where}~  t \in [t_0,t_0+N]~. 
\end{equation}
Thus, this definition is adapted to the truncated \( \ell^\infty\)-cone \( \KN\), which is  given by
\begin{equation}\label{prelim:eq_truncated_cone}
K_{t_0,x_0}^N := \{ (t,x) \in [t_0,t_0+N] \times \rfour \colon \|x-x_0\|_\infty \leq 2N-|t-t_0| \} ~. 
\end{equation}
It might be more appropriate to call \( K_{t_0,x_0}^N \)  a pyramid (see Figure \ref{figure:linf}); unfortunately, the letter \( P \) is already heavily used in our notation, so that we decided to use the letter \( K \). Our reason for using the \( \ell^\infty\)-norm, instead of the more common \( \ell^2\)-norm, lies in the induction on scales argument (Proposition \ref{prelim:eq_total_energy_increment}). Then, it will be an advantage to write \( \KN \) as the union of finitely overlapping smaller cones \( K_{\tau_0,y_0}^{M} \), which are contained in \( \KN \). Using finite speed of propagation and the inequality \( \| \cdot \|_{\ell^\infty(\rfour)} \leq \| \cdot \|_{\ell^2(\rfour)} \), one can still meaningfully restrict the nonlinear wave equation to \( \KN \). 

\begin{figure}[t!]
\begin{center}
\includegraphics[height=6cm]{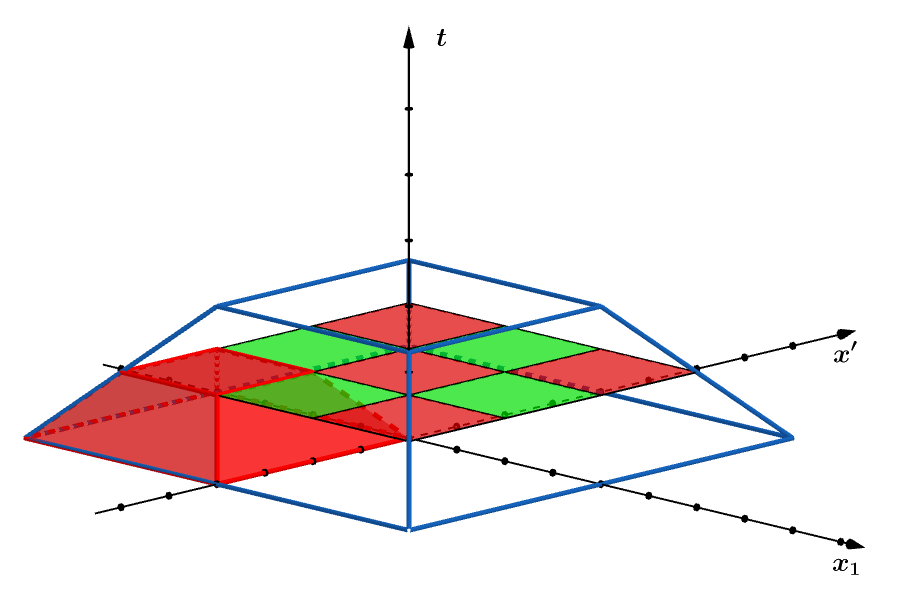}
\end{center}
\caption*{\small{In this figure, we illustrate the truncated \( \ell^\infty\)-cone, and its decomposition into smaller cones. The blue lines are the edges of a large \( \ell^\infty\)-cone.  The red and green squares are the tops of smaller truncated cones. In the lower left corner, we have drawn a single one of these smaller cones. As can be seen from this figure, no smaller cone has to exit the large cone. }}
\caption{Decomposition of the \( \ell^\infty\)-cone}
\label{figure:linf}
\end{figure}

\begin{lem}[Local energy increment]\label{prelim:ref_local_energy_increment}
Let \( v \) be a solution to the forced equation \eqref{prelim:eq_forced_nlw}, let \( N \geq 1 \),  let 
\( t_0 \in N \mathbb{N}_0 \), and let \( x_0 \in N \zfour\). Then, we have that 
\begin{equation}\label{prelim:eq_local_energy_increment}
\sup_{t \in [t_0,t_0+N]} E_{t_0,x_0}^N[v](t) \leq E_{t_0,x_0}^N[v](t_0) + 6 \int_{\KN}  |F| |v|^2 |\partial_t v| \dx\dt + 3 \int_{\KN}  |F|^3 |\partial_t v| \dx\dt~. 
\end{equation}
\end{lem}

\begin{proof}
Using \eqref{prelim:eq_energy_flux}, we have that 
\begin{align*}
\frac{\mathrm{d}}{\dt} \En[v](t) &= - \int \displaylimits_{ \substack{\| x - x_0 \|_\infty \\= 2N- |t-t_0|}} T^{00}\dsigma(x) + \int\displaylimits_{ \substack{\| x - x_0 \|_\infty \\\leq 2N- |t-t_0|}}  \partial_t T^{00}\dx \\
&=   \int\displaylimits_{ \substack{\| x - x_0 \|_\infty \\= 2N- |t-t_0|}}  (-T^{00}+T^{0j} \nu_j ) \dsigma(x) - \int\displaylimits_{ \substack{\| x - x_0 \|_\infty \\ \leq2N- |t-t_0|}}  \mathcal{N} \partial_t v \dx
\end{align*}
Here, \( \nu \) is the outward unit normal to the cube. From Cauchy-Schwarz, it follows that \( |T^{0j}\nu_j |\leq |\partial_t v| \| \nabla v \|_2 \leq T^{00} \). After integrating in time, this completes the proof.
\end{proof}
To simplify the notation, we now write 
\begin{equation}
\EN[v] := \sup_{t \in [t_0,t_0+N]} E_{t_0,x_0}^N[v](t) ~. 
\end{equation}

In the following, we want to deduce a flux estimate for the solution \( v \) of the forced NLW \eqref{prelim:eq_forced_nlw}. Here, we encounter a minor technical problem. Let \( (t^\prime,x^\prime) \in \KN \)  be a point in the truncated \(\ell^\infty\)-cone. We then want to control the potential \( |v|^4 \) on the truncated light-cone \begin{equation*}
C^N_{t^\prime,x^\prime}:= \{ (t,x)\in [t_0,t_0+N]\times \rfour \colon |t-t^\prime | = \|x-x^\prime \|_{\small 2}\}~. 
\end{equation*}
Unfortunately, \( C^N_{t^\prime,x^\prime} \) may not be contained in \( \KN \), and hence we cannot expect to bound this solely by \( \EN[v] \). Since the flux estimate is derived through a monotonicity formula for the local energy, this issue persists even if we are only interested in the portion of \( C^N_{t^\prime,x^\prime} \) intersecting \( \KN \). To solve this problem, while still keeping the same energy increment as in \eqref{prelim:eq_local_energy_increment}, we introduce the notion of a locally forced solution.

\begin{definition}[Locally forced solution]\label{prelim:def_locally_forced_solution}
Let \( t_0 \in N\mathbb{N}_0 \) and \( x_0 \in N \zfour \). We call \( w \colon \reals \times \rfour \rightarrow \reals \) a \( \KN \)-locally forced solution if it solves
\begin{equation}\label{prelim:eq_locally_forced}
\begin{cases}
-\partial_{tt} w + \Delta w = (1_{\KN} F+w)^3 \qquad (t,x)\in \reals\times \reals^4\\
w|_{t=t_0}  =w_0 \in \dot{H}^1(\rfour), \quad \partial_t w|_{t=t_0} = w_1\in L^2(\rfour) ~
\end{cases}~.
\end{equation}
We also require that the functions  \( (w_0,w_1)\) agree with \( (v(t_0),\partial_t v(t_0)) \) on the cube \( \|x-x_0\|_\infty\leq 2N \).  
\end{definition}

\begin{rem}
From finite speed of propagation, it follows that \( w|_{\KN} = v|_{\KN} \). 
\end{rem}
For the same reasons as described in the last paragraph, we also use the energy on a slightly larger region. To this end, we define
\begin{equation*}
\widetilde{E}_{t_0,x_0}^N[w](t) := \int_{\|x-x_0 \|_\infty \leq 16N-|t-t_0|} T^{00}[w](t,x) \dx, \qquad \text{where}~  t \in [t_0,t_0+N]~. 
\end{equation*}
Thus, this definition is adapted to the fattened cone \( \KNtilde \), which is defined in \eqref{wp:eq_fattened_cone}. 
We also set 
\begin{equation}
\ENtilde[w]:= \sup_{t\in [t_0,t_0+N]} \widetilde{E}_{t_0,x_0}^N[w](t) ~. 
\end{equation}

\begin{lem}[Local flux estimate]\label{prelim:lem_local_flux}
Let \( t_0 \in N\mathbb{N}_0 \), let \( x_0 \in N \zfour \), and let \( w \) be  a \( \KN \)-locally forced solution. Then, we have that 
\begin{equation}\label{prelim:eq_local_flux}
\sup_{\substack{ t^\prime \in [t_0,t_0+N] \\ \| x^\prime-x_0 \|_{\ell^\infty}\leq 8N}} ~ \int \displaylimits_{\substack{\|x-x^\prime\|_2= |t-t^\prime| \\ t\in[t_0,t_0+N]}} \frac{w^4}{4} ~ \dsigma(t,x) \leq 4 \ENtilde + 12 \int\displaylimits_{\KN} |F| (|w|+|F|)^2 |\partial_t w| \dx \dt ~. 
\end{equation}
\end{lem} 

We emphasize that, even though the energy \( \ENtilde \) is measured on a truncated \( \ell^\infty \)-cone, the flux is still controlled on a light cone. 
The estimate \eqref{prelim:eq_local_flux}, however, only controls \( w \) on a lower-dimensional surface in space-time, and thus cannot directly be used to bound the energy increment. In our main argument, we rely on the following averaged version. 

\begin{lem}[Averaged local flux estimate]\label{prelim:lem_averaged_flux} Let \( N\geq N_{\hi}  \), let \( t_0 \in N\mathbb{N}_0 \), let \( x_0 \in N \zfour \), and let \( w \) be  a \( \KN \)-locally forced solution. Then, we have that 
\begin{align}
\FluxNtilde&:= 
\sup_{\substack{ t^\prime \in [t_0,t_0+N]\\ \|x^\prime -x_0\|_\infty\leq 4N}} ~\int \displaylimits_{\substack{|\|x-x^\prime\|_2- |t-t^\prime||\leq N^{10\delta} \\ t\in[t_0,t_0+N]}} \frac{w^4}{4} ~\dx\dt \label{prelim:eq_flux} \\
&\lesssim N^{40\delta} \left( \ENtilde +  \int\displaylimits_{\KN} |F| (|w|+|F|)^2 |\partial_t w| \dx \dt \right)~. \notag
\end{align}
\end{lem}

The appeareance of \( N^{10\delta}\) is for technical reasons only, and the reader is encouraged to mentally replaced it by \( 1 \). This will later help us to deal with the spatial tails of the wave packets. 

\begin{proof}[Proof of Lemma \ref{prelim:lem_local_flux}] 
For the duration of this proof, we define
\begin{equation*}
e(t):= \int_{\| x - x^\prime\|_2 \leq |t-t^\prime|} T^{00}(t,x) \dx 
\end{equation*}
From finite speed of propagation, we expect \( e(t) \) to be (nearly) non-increasing on \( [t_0,t^\prime] \) and non-decreasing on \( [t^\prime,t_0+N] \). From the assumptions above, it follows that \( {\| x - x^\prime\|_2 \leq |t-t^\prime|}  \) implies
\begin{equation*}
\| x-x_0 \|_{\infty} \leq \| x-x^\prime \|_{\infty} + \| x^\prime - x_0 \|_\infty \leq \| x-x^\prime \|_2 + 8N \leq |t-t^\prime| + 8N \leq 10N - |t-t_0|~. 
\end{equation*}
Thus, we obtain that \( e(t) \leq \ENtilde \) for all \( t \in [t_0,t_0+N] \).  Using \eqref{prelim:eq_energy_flux}, we obtain for all \( t \in [t^\prime,t_0+N] \) that 
\allowdisplaybreaks
\begin{align*}
\frac{\mathrm{d}}{\dt} e(t) &= \int_{\| x - x^\prime\|_2 \leq |t-t^\prime|} \partial_t T^{00}(t,x) \dx + \int_{\| x - x^\prime\|_2 = |t-t^\prime|} T^{00}(t,x) \dsigma(x)\\
&= - \int_{\| x - x^\prime\|_2 \leq |t-t^\prime|} \mathcal{N} \partial_t v \dx - \int_{\| x - x^\prime\|_2 \leq |t-t^\prime|} \partial_{x_k} T^{0k} \dx + \int_{\| x - x^\prime\|_2 = |t-t^\prime|} T^{00}(t,x) \dsigma(x) \\
&= - \int_{\| x - x^\prime\|_2 \leq |t-t^\prime|} \mathcal{N} \partial_t v \dx + \int_{\| x - x^\prime\|_2 = |t-t^\prime|} (T^{00}(t,x) + T^{0k} \nu_k) \dsigma(x) \\
&\geq -6 \int_{\| x - x^\prime\|_2 \leq |t-t^\prime|} |F| ( |F|+|v|)^2 |\partial_t v|\dx + \int_{\| x - x^\prime\|_2 = |t-t^\prime|} \frac{v^4}{4} \dsigma(x)~. 
\end{align*}
Integrating this inequality in time, we obtain the result on \( [t^\prime,t_0+N] \). The bound on \( [t_0,t^\prime] \) is similar.
\end{proof}

\begin{proof}[Proof of Lemma \ref{prelim:lem_averaged_flux}] Since \( N\geq N_{\hi} \geq \eta^{-1} \), we have that \( N^{10\delta} \ll N \). 
We then simply integrate \eqref{prelim:eq_local_flux} over a spatial ball of radius \( \sim N^{10\delta} \) around \( x^\prime \). 
\end{proof}

\section{The energy increment and induction on scales}\label{section:a_priori}

We are now ready to finally bound the energy increment of the nonlinear component \( v \). The argument roughly splits into two parts: A single scale analysis and induction on scales. 

For technical reasons, we define a flux-term involving a thinner neighborhood of the cone. More precisely, we let 
\begin{equation*}
\FluxN[w] :=\sup_{\substack{(t^\prime,x^\prime)\in [t_0,t_0+N]\times \rfour\\ \|x^\prime -x_0\|_\infty\leq 4N}} \int \displaylimits_{\substack{|\|x-x^\prime\|_2- |t-t^\prime||\leq N^{5\delta} \\ t\in[t_0,t_0+N]}} \frac{w^4}{4} ~\dx\dt \\
\end{equation*}
Recall that the light-cone in \( \FluxNtilde[w] \), as defined in \eqref{prelim:eq_flux}, has width \( N^{10\delta} \).

\begin{prop}[Single-scale energy increment] \label{apriori:prop_single_scale}
Let \( N \geq 1 \). Let \( t_0 \in N \mathbb{N}_0 \),  where \( 0 \leq t_0 /N \leq \lfloor N^\theta \rfloor\), and let \( x_0 \in N \zfour \). Furthermore, let \( w_1,w_2 \in L_t^\infty L_x^4(\reals\times \rfour) \) and \( w_3 \in L_t^\infty L_x^2(\reals\times \rfour) \). Then, 
\begin{align*}
&\int_{\KN} |F_M| |w_1| |w_2| |w_3| \dx \dt \\
& \lesssim \eta N^{\frac{3}{4}-s+8\delta}  ( \| w_1 \|_{L_t^\infty L_x^4(\reals\times \rfour)}^4 + \FluxN[w_1] )^{\frac{1}{4}}  ( \| w_2 \|_{L_t^\infty L_x^4(\reals\times \rfour)}^4 + \FluxN[w_2] )^{\frac{1}{4}} \| w_3 \|_{L_t^\infty L_x^2(\reals\times \rfour)} ~.
\end{align*}
\end{prop}

We have two separate reasons for introducing the auxiliary functions \( w_1\), \( w_2 \), and \( w_3\). First, it emphasizes that the proof does not depend on the evolution equation of the nonlinear component. Second, it allows us to pass to smaller spatial scales than \( N \) with minimal notational effort, see Corollary \ref{apriori:cor_coarse_scale}.

\begin{proof}
If \( 1 \leq N < N_{\hi} \), there is nothing to show. Thus, we may assume that \( N \geq N_{\hi} \). 

\emph{Step 1: Wave packet decomposition.} Recall from \eqref{wp:eq_half_wave_decomposition} and \eqref{wp:eq_FN} that 
\begin{align*}
F_N^\omega = \sum_{k} \epsilon_k \exp(it|\nabla|) f_k^+ +  \sum_{k} \epsilon_k \exp(-it|\nabla|) f_k^-~. 
\end{align*}
We only control the contribution of \(  \sum_{k} \epsilon_k \exp(it|\nabla|)  f_k^+ \), the other estimate is nearly identical. Then, we may also drop the superscript \( + \) from our notation. We now apply 
Proposition \ref{wp:prop_wave_packet_decomposition} and Proposition \ref{wp:prop_wp_bushes} to the family \( \{ \epsilon_k f_k \}_k \), and let the sets \( \Am, \Bj, \) and \( \Dm \) be as in Proposition \ref{wp:prop_wp_bushes}. 
As before, we implicitly restrict to \( \| k \|_\infty \in (N/2,N] \). 
We also write \( f_{k,l} = f_{k,l;t_0} \).

\emph{Step 2: Distant wave packets and extreme amplitudes.} On a heuristic level, the wave packets whose tubes \( T_{k,l} \) do not intersect \( \KN \) should not contribute to the integral. We now make this precise using the decay estimate \eqref{wp:eq_decay}. Indeed,
\begin{align*}
&\int_{\KN} \bigg| \sum\displaylimits_{\substack{ (k,l)\in \zfour\times \zfour:\\ \| l - x_0 \|_\infty > 3N}} \epsilon_k \exp(i(t-t_0)|\nabla|) f_{k,l} \bigg| |w_1| |w_2| |w_3| \dx \dt \\
&\lesssim N \Big\|  \sum\displaylimits_{\substack{ (k,l)\in \zfour\times \zfour:\\ \| l - x_0 \|_\infty > 3N}} \epsilon_k \exp(i(t-t_0)|\nabla|) f_{k,l} \Big  \|_{L_t^\infty L_x^\infty(\KN)} \| w_1 \|_{L_t^\infty L_x^4(\reals\times \rfour)} 
 \| w_2 \|_{L_t^\infty L_x^4(\reals\times \rfour)}  \| w_3 \|_{L_t^\infty L_x^2(\reals\times \rfour)} \\
 &\lesssim N \Big ( \sum\displaylimits_{\substack{ (k,l)\in \zfour\times \zfour:\\ \| l - x_0 \|_\infty > 3N}} N^{-100} (1+\| x_0 - l\|_2)^{-100} \| f_{k} \|_{L_x^2(\rfour)} \Big )\| w_1 \|_{L_t^\infty L_x^4(\reals\times \rfour)} 
 \| w_2 \|_{L_t^\infty L_x^4(\reals\times \rfour)}  \| w_3 \|_{L_t^\infty L_x^2(\reals\times \rfour)} \\
 &\lesssim N^{-99+2} \Big ( \sum\displaylimits_{k\in \zfour} \| f_{k} \|_{L_x^2(\rfour)}^2 \Big)^{\frac{1}{2}} \| w_1 \|_{L_t^\infty L_x^4(\reals\times \rfour)} 
 \| w_2 \|_{L_t^\infty L_x^4(\reals\times \rfour)}  \| w_3 \|_{L_t^\infty L_x^2(\reals\times \rfour)}  \\
 &\lesssim \eta N^{\frac{3}{4}-s} \| w_1 \|_{L_t^\infty L_x^4(\reals\times \rfour)} 
 \| w_2 \|_{L_t^\infty L_x^4(\reals\times \rfour)}  \| w_3 \|_{L_t^\infty L_x^2(\reals\times \rfour)} ~. 
\end{align*}
Thus, this contribution is acceptable. It remains to control the wave packets with indices in \( \mathsmaller{\bigcup}_{m\in \mathbb{Z}} \Am \). We now use crude estimates to reduce to \( \sim \log(N) \) amplitude scales. Let \( m \leq -20 \log(N) \). Since \( \# \Am \lesssim N^{8} \), we have that 
\begin{align*}
&\int_{\KN} \Big| \sum_{(k,l)\in \Am} \epsilon_k \exp(i(t-t_0)|\nabla|) f_{k,l} \Big| |w_1| |w_2| |w_3|\dx \dt \displaybreak[1] \\
&\lesssim N \Big ( \sum_{(k,l)\in \Am} \| \exp(i(t-t_0)|\nabla|) f_{k,l} \|_{L_t^\infty L_x^\infty(\reals\times \rfour)} \Big)  \| w_1 \|_{L_t^\infty L_x^4(\reals\times \rfour)} 
 \| w_2 \|_{L_t^\infty L_x^4(\reals\times \rfour)}  \| w_3 \|_{L_t^\infty L_x^2(\reals\times \rfour)}  \displaybreak[1]  \\
 &\lesssim N2^m \# \Am  \| w_1 \|_{L_t^\infty L_x^4(\reals\times \rfour)} 
 \| w_2 \|_{L_t^\infty L_x^4(\reals\times \rfour)}  \| w_3 \|_{L_t^\infty L_x^2(\reals\times \rfour)}   \displaybreak[1] \\
 &\lesssim 2^m N^9  \| w_1 \|_{L_t^\infty L_x^4(\reals\times \rfour)} 
 \| w_2 \|_{L_t^\infty L_x^4(\reals\times \rfour)}  \| w_3 \|_{L_t^\infty L_x^2(\reals\times \rfour)} ~. 
\end{align*}
Summing this inequality over all \( m \leq -20\log(N) \), we obtain that 
\begin{align*}
&\int_{\KN} \Big| \sum_{m\leq -20\log(N)} \sum_{(k,l)\in \Am} \epsilon_k \exp(i(t-t_0)|\nabla|) f_{k,l} \Big| |w_1| |w_2| |w_3|\dx \dt \displaybreak[1]\\
&\lesssim N^{-11} \| w_1 \|_{L_t^\infty L_x^4(\reals\times \rfour)} 
 \| w_2 \|_{L_t^\infty L_x^4(\reals\times \rfour)}  \| w_3 \|_{L_t^\infty L_x^2(\reals\times \rfour)}\displaybreak[1]\\
 &\lesssim \eta N^{\frac{3}{4}-s} \| w_1 \|_{L_t^\infty L_x^4(\reals\times \rfour)} 
 \| w_2 \|_{L_t^\infty L_x^4(\reals\times \rfour)}  \| w_3 \|_{L_t^\infty L_x^2(\reals\times \rfour)} ~.  
\end{align*}
Finally, if \( \Am \neq \emptyset \), then \( \# \Am \geq 1 \). This implies that 
\begin{equation*}
2^m \leq 2^m (\# \Am)^{\frac{1}{2}} \lesssim \left( \sum_{k,l} \| f_{k,l} \|_{L^2(\rfour)}^2 \right)^{\frac{1}{2}} \lesssim \eta N^{-s}~. 
\end{equation*}
For a sufficiently small absolute constant \( \eta > 0\), this implies that \( m \leq 0\). This completes the crude part of the argument. 
\emph{Step 3: Bushes.}
First, we define the fattened tubes by 
\begin{equation*}
\widetilde{T}_{k,l}  := \{ (t,x)\in [t_0,t_0+N] \times \reals^d \colon | x- ( l - t \cdot k/\|k\|_2 ) |\leq N^{2\delta} \}~.\\
\end{equation*}
Furthermore, we define the collection of fattened tubes corresponding to a bush by 
\begin{equation*}
\widetilde{T}(\Bj) := \mathsmaller{\bigcup}_{(k,l)\in \Bj}~ \widetilde{T}_{k,l}~. 
\end{equation*}
With these definitions in hand, we now write 
\begin{align*}
&\int_{\KN} \Big| \sum_{j=1}^{J} \sum_{(k,l)\in \Bj} \epsilon_k  \exp(i(t-t_0)|\nabla|)f_{k,l} \Big| |w_1| |w_2| |w_3| \dx \dt \\
&\leq \sum_{j=1}^{J} \int_{\Tbj}  \Big| \sum_{(k,l)\in \Bj} \epsilon_k  \exp(i(t-t_0)|\nabla|) f_{k,l} \Big|  |w_1| |w_2| |w_3|\dx \dt \\
&+\sum_{j=1}^{J} \int_{\KN\backslash\Tbj}  \Big| \sum_{(k,l)\in \Bj} \epsilon_k \exp(i(t-t_0)|\nabla|) f_{k,l} \Big| |w_1| |w_2| |w_3|\dx \dt 
\end{align*}
Using that all tubes in \( \Bj \) pass through the same space-time cube of size \( \sim N^{\delta} \), we obtain from Proposition \ref{wp:prop_smallness} that 
\begin{align*}
&  \sum_{j=1}^{J} \int_{\Tbj}  \Big| \sum_{(k,l)\in \Bj} \epsilon_k  \exp(i(t-t_0)|\nabla|) f_{k,l} \Big|  |w_1| |w_2| |w_3|\dx \dt\\
& \lesssim \sum_{j=1}^{J}  N^{\frac{1}{2}} \Big\| \sum_{(k,l)\in \Bj} \epsilon_k  \exp(i(t-t_0)|\nabla|) f_{k,l} \Big\|_{L_{t,x}^\infty(\reals\times \rfour)} \| w_1\|_{L_{t,x}^4(\Tbj)}  \| w_2 \|_{L_{t,x}^4(\Tbj)}\| w_3 \|_{L_t^\infty L_x^2(\reals\times \rfour)}\displaybreak[2]\\
&\lesssim \eta N^{\frac{1}{2}+8\delta} 2^m  \Big( \sum_{j=1}^{J} (\# \Bj)^{\frac{1}{2}} \Big)  \FluxN[w_1]^{\frac{1}{4}} \FluxN[w_2]^{\frac{1}{4}}\| w_3 \|_{L_t^\infty L_x^2(\reals\times \rfour)} \displaybreak[2]\\
&\lesssim \eta N^{\frac{1}{2}+8\delta}2^{m}  \Big( \sum_{j=1}^J \Bj \Big) \mu^{-\frac{1}{2}}  \FluxN[w_1]^{\frac{1}{4}} \FluxN[w_2]^{\frac{1}{4}}\| w_3 \|_{L_t^\infty L_x^2(\reals\times \rfour)} \\
&\lesssim \eta N^{\frac{1}{2}+8\delta}2^{m}  (\# \Am) \mu^{-\frac{1}{2}}    \FluxN[w_1]^{\frac{1}{4}} \FluxN[w_2]^{\frac{1}{4}}\| w_3 \|_{L_t^\infty L_x^2(\reals\times \rfour)}  ~. 
\end{align*}
Here, \( \mu \) denotes the minimum number of packets inside a single bush, see Proposition \ref{wp:prop_wp_bushes}. 
Using the decay estimate \eqref{wp:eq_decay}, we control the contributions outside the \( \Tbj \) by 
\begin{align*}
&\sum_{j=1}^{J} \int_{\KN\backslash \Tbj }  \Big| \sum_{(k,l)\in \Bj} \epsilon_k  \exp(i(t-t_0)|\nabla|) f_{k,l} \Big| |w_1| |w_2| |w_3|\dx \dt \\
&\lesssim N \sum_{j=1}^{J}   \Big\| \sum_{(k,l)\in \Bj} \epsilon_k \exp(i(t-t_0)|\nabla|) f_{k,l} \Big\|_{L_{t,x}^\infty({\KN\backslash \Tbj)}} \| w_1 \|_{L_t^\infty L_x^4(\reals\times \rfour)} \| w_2 \|_{L_t^\infty L_x^4(\reals\times \rfour)} \| w_3\|_{L_t^\infty L_x^2(\reals\times \rfour)} \\
&\lesssim N^1 N^{-100} 2^m \Big( \sum_{j=1}^{J} \# \Bj\Big)\| w_1 \|_{L_t^\infty L_x^4(\reals\times \rfour)} \| w_2 \|_{L_t^\infty L_x^4(\reals\times \rfour)} \| w_3\|_{L_t^\infty L_x^2(\reals\times \rfour)} \\
&\lesssim \eta N^{\frac{3}{4}-s} \| w_1 \|_{L_t^\infty L_x^4(\reals\times \rfour)} \| w_2 \|_{L_t^\infty L_x^4(\reals\times \rfour)} \| w_3\|_{L_t^\infty L_x^2(\reals\times \rfour)} ~. 
\end{align*}
In the last line, we have used that \(2^m ( \# \Am)^{\frac{1}{2}} \lesssim \eta  \) and \( \# \Am \lesssim N^8 \).

\emph{Step 4: Disjoint wave packets.} We now control the contribution of the almost disjoint family \( \Dm \). If \( \mu< 1 \), then \( \Dm \) is empty, and there is nothing to prove.  If \( \mu \geq 1 \), it follows from Proposition \ref{wp:prop_smallness} that
\begin{align*}
&\int_{\KN} \Big| \sum_{(k,l)\in \Dm} \epsilon_k \exp(i(t-t_0)|\nabla|) f_{k,l} \Big| |w_1| |w_2| |w_3| \dx \dt \\
&\lesssim N  \Big\| \sum_{(k,l)\in \Dm} \epsilon_k  \exp(i(t-t_0)|\nabla|) f_{k,l} \Big\|_{L_{t,x}^\infty(K^N_{x_0})} \| w_1 \|_{L_t^\infty L_x^4(\reals\times \rfour)} \| w_2 \|_{L_t^\infty L_x^4(\reals\times \rfour)} \| w_3\|_{L_t^\infty L_x^2(\reals\times \rfour)}\\
&\lesssim \eta N^{1+8\delta} 2^m  \mu^{\frac{1}{2}}  \| w_1 \|_{L_t^\infty L_x^4(\reals\times \rfour)} \| w_2 \|_{L_t^\infty L_x^4(\reals\times \rfour)} \| w_3\|_{L_t^\infty L_x^2(\reals\times \rfour)}. 
\end{align*}

\emph{Step 5: Finishing the proof.}

In total, we have shown that 
\begin{equation}\label{apriori:eq_choice_P}
\begin{aligned}
&\int_{\KN} |F_M| |w_1| |w_2| |w_3| \dx \dt \\
&\lesssim  \eta N^{8\delta}\left( N^{\frac{3}{4}-s} + N 2^m  \mu^{\frac{1}{2}}  + N^{\frac{1}{2}}2^{m}  (\# \Am) \mu^{-\frac{1}{2}}  \right)  ( \| w_1 \|_{L_t^\infty L_x^4(\reals\times \rfour)}^4 + \FluxN[w_1] )^{\frac{1}{4}} \\
&\quad \cdot ( \| w_2 \|_{L_t^\infty L_x^4(\reals\times \rfour)}^4 + \FluxN[w_2] )^{\frac{1}{4}} \| w_3 \|_{L_t^\infty L_x^2(\reals\times \rfour)}~~.
\end{aligned}
\end{equation}
Due to our choice \( \mu:= N^{-\frac{1}{2}} \# \Am \) in Proposition \ref{wp:prop_wp_bushes}, this completes the proof. 
\end{proof}

\begin{cor}[Coarse-scale energy increment]\label{apriori:cor_coarse_scale} 
Let \( N \geq 1 \). Let \( t_0 \in N \mathbb{N}_0 \), with \( 0\leq t_0/N \leq \lfloor N^\theta \rfloor \), and let \( x_0 \in N \zfour \). 
Let \( w\) be a \( \KN \)-locally forced solution. Then, we have for all \( M \geq N \) that 
\begin{equation}\label{apriori:eq_coarse_scale}
\int_{\KN} |F_M| |w|^2 |\partial_t w |\dx \dt \lesssim \eta M^{\frac{3}{4}-s+8\delta} \ENtilde[w]^{\frac{1}{2}} \left( \ENtilde[w] + \FluxNtilde[w] \right)^{\frac{1}{2}}~. 
\end{equation}
\end{cor}
We refer to Corollary \ref{apriori:cor_coarse_scale} as a coarse scale estimate since the wave packets in \( F_M \) are atleast as long as the length of \( \KN \) (in time).  

\begin{proof}
As before, we may take \( M\geq N_{\hi} \). We distinguish two different cases. If \( M \geq N^{\frac{4}{3}} \), we obtain from Proposition \ref{wp:prop_smallness} that
\begin{equation*}
\int_{\KN} |F_M| |w|^2 |\partial_t w| \dx \dt \leq N \| F_M \|_{L_t^\infty L_x^\infty(\reals\times \rfour)} \| w \|_{L_t^\infty L_x^4(\KN)}^2 \| \partial_t w \|_{L_t^\infty L_x^2(\KN)} 
\lesssim \eta M^{\frac{3}{4}-s+\delta} \ENtilde[w]~. 
\end{equation*}
Next, we let \( M \leq N^{\frac{4}{3}} \). Then, there exist \( \tau_0 \in M\mathbb{N}_0 \) and \( y_0 \in M\zfour \) s.t. \( \KN \subseteq K_{\tau_0,y_0}^M \). Furthermore, since \( t_0/N\leq \lfloor N^\theta \rfloor\), it holds that  \( \tau_0/M \leq \lfloor M^\theta \rfloor \). Set \( w_1 = w_2  := 1_{\KN} w \) and \( w_3 := 1_{\KN} w \). Since \( M \leq N^{\frac{4}{3}} \), we have that 
\begin{equation*}
\mathcal{F}_{\tau_0,y_0}^M[w_1] = \mathcal{F}_{\tau_0,y_0}^M[w_2] \leq \widetilde{\mathcal{F}}_{t_0,x_0}^N[w]
\end{equation*}
Using the single-scale energy increment (Proposition \ref{apriori:prop_single_scale}), we obtain that
\begin{align*}
&\int_{\KN} |F_M| |w|^2 |\partial_t w|\dx \dt \\
 &= \int_{K_{\tau_0,y_0}^M} |F_M| |w_1| |w_2| |w_3|\dx \dt\\
 &\lesssim \eta M^{\frac{3}{4}-s+\delta} ( \| w_1 \|_{L_t^\infty L_x^4(\reals\times \rfour)}^4 + \FluxM[w_1] )^{\frac{1}{4}}  ( \| w_2 \|_{L_t^\infty L_x^4(\reals\times \rfour)}^4 + \FluxM[w_2] )^{\frac{1}{4}} \| w_3 \|_{L_t^\infty L_x^2(\reals\times \rfour)} \\
 &\lesssim \eta M^{\frac{3}{4}-s+\delta} \ENtilde[w]^{\frac{1}{2}} \left( \ENtilde[w] + \FluxNtilde[w] \right)^{\frac{1}{2}}~.
\end{align*}
\end{proof}

Due to Proposition \ref{apriori:prop_single_scale} and Corollary \ref{apriori:cor_coarse_scale}, we understand the energy increment at a single scale. Unfortunately, the cone \( \KN \) may contain many wave packets on smaller scales. Similar problems are often encountered in restriction theory, and can sometimes be solved using Wolff's induction on scales strategy \cite{Wolff01}. The following argument can be seen as a (simple) implementation of this idea.

\begin{prop}[Induction on scales]\label{apriori:prop_induction_scales}
Let \( s> \max(1-\theta/2,3/4+\theta) \). Let \( R \geq 1 \) be a dyadic integer, and let \( F \) be as in Proposition \ref{wp:prop_smallness}. Let \( t_0 \in R\mathbb{N}_0 \), with \( t_0/R \leq \lfloor R^{\theta} \rfloor\),  \( x_0 \in R\zfour \), let \( w \) be a \( \KR \)-locally forced solution. For a large absolute constant \( C_1 \geq 1 \), we have that
\begin{equation}\label{apriori:eq_induction_energy}
\ERtilde[w]  \leq 2 E_{|x-x_0|\leq 16R}[w](t_0) + C_1\| F \|_{L_t^3 L_x^6(\KR)}^6
\end{equation}
and 
\begin{equation}\label{apriori:eq_induction_flux}
\FluxRtilde[w] \leq C_1 R^{50\delta} \left(  E_{|x-x_0|\leq 16R}[w](t_0) +  \| F \|_{L_t^3 L_x^6(\KR)}^6 \right)~. 
\end{equation}
\end{prop}

\begin{proof}
We use induction on the dyadic integers \( R \geq 1 \). \\

\emph{Step 1: Base case \( R=1 \).} We have that 
\begin{align*}
\Eonetilde[w] &\leq E_{|x-x_0|\leq 16}[w](t_0) + C \int_{\Kone} |F| |w|^2 |\partial_t w|\dx \dt + C \int_{\Kone} |F|^3 |\partial_t w |\dx \dt \\
&\leq E_{|x-x_0|\leq 16}[w](t_0) + C \| F \|_{L_t^\infty L_x^\infty(\reals\times \rfour)} \| w \|_{L_t^\infty L_x^4(\Kone)}^2 \| \partial_t w\|_{L_t^\infty L_x^2(\Kone)} \\
&~~~+ C \| F \|_{L_t^3 L_x^6(\Kone)}^3 \| \partial_t w \|_{L_t^\infty L_x^2(\Kone)} \\
&\leq  E_{|x-x_0|\leq 16}[w](t_0) + \frac{1}{4} \Eonetilde[w] + C \| F \|_{L_t^3 L_x^6(\Kone)}^3 \Eonetilde[w]^{\frac{1}{2}} \\
&\leq E_{|x-x_0|\leq 16}[w](t_0)  +  \frac{1}{2} \Eonetilde[w] + C \| F \|_{L_t^3 L_x^6(\Kone)}^6~. 
\end{align*}
 Insert this bound into Lemma \ref{prelim:lem_averaged_flux}, we obtain that 
\begin{align*}
\Fluxonetilde[w] &\lesssim \Eonetilde[w] +  \int_{\Kone} |F| |w|^2 |\partial_t w|\dx \dt + C \int_{\Kone} |F|^3 |\partial_t w |\dx \dt  \\
				&\lesssim \Eonetilde[w] \\
				&\lesssim E_{|x-x_0|\leq 16}[w](t_0) + \| F \|_{L_t^3 L_x^6(\Kone)}^6~.
\end{align*}
By choosing \( C_1 \) sufficiently large, we obtain \eqref{apriori:eq_induction_energy} and \eqref{apriori:eq_induction_flux}. This already determines our choice of \( C_1 \), which we now regard as a fixed constant. Let \( R\geq 2 \) be an arbitrary dyadic integer. Using the induction hypothesis, we can rely on the inequalities \eqref{apriori:eq_induction_energy} and \eqref{apriori:eq_induction_flux} for all scales \( N \leq R/2 \). \\

\emph{Step 2: Splitting the energy increment.} 
From Lemma \ref{prelim:ref_local_energy_increment}, we have that 
\begin{equation}\label{apriori:eq_proof_energy_increment}
\ERtilde[w]\leq E_{|x-x_0|\leq 16R}[w](t_0) + C \int_{\KR} |F| |w|^2 |\partial_t w| \dx \dt + C \int_{\KR} |F|^3 |\partial_t w | \dx \dt~. 
\end{equation}
The main term is the second summand in \eqref{apriori:eq_proof_energy_increment}. We use a Littlewood-Paley type decomposition of the linear evolution and write 
\begin{align*}
 &\int_{\KR} |F| |w|^2 |\partial_t w| \dx \dt\\
 &\leq \sum_{N\geq R} \int_{\KR} |F_N| |w|^2 |\partial_t w| \dx \dt + \sum_{N\leq R/2} \int_{\KR} |F_N| |w|^2 |\partial_t w| \dx \dt~. 
\end{align*}
\emph{Step 3: High frequencies.}
The high frequencies can be controlled using the single-scale estimate from Proposition \ref{apriori:prop_single_scale}. Indeed, we have that 
\begin{align}
 &\sum_{N\geq R} \int_{\KR} |F_N| |w|^2 |\partial_t w| \dx \dt \notag \\
 &\lesssim \eta \sum_{N\geq R} N^{\frac{3}{4}-s+8\delta} \ERtilde[w]^{\frac{1}{2}} \left( \ERtilde[w] + \FluxRtilde[w] \right)^{\frac{1}{2}}\label{apriori:eq_proof_high_frequencies}  \\
 &\lesssim \eta R^{\frac{3}{4}-s+8\delta} \ERtilde[w]^{\frac{1}{2}} \left( \ERtilde[w] + \FluxRtilde[w] \right)^{\frac{1}{2}}~. \notag
\end{align}
\emph{Step 4: Low frequencies.} 
For \( \tau_n = Nn \in N\mathbb{N}_0 \) and \( y_j = Nj \in N\zfour\) , we write
\begin{align*}
&\int_{\KR} |F_N| |w|^2 |\partial_t w| \dx \dt \\
&= \sum_{n=0}^{\lfloor N^\theta \rfloor} \bigg( \int\displaylimits_{\substack{([\tau_n,\tau_n+N]\times \rfour)  \mathsmaller{\bigcap} \KR} }|F_N| |w|^2 |\partial_t w| \dx \dt \bigg)
+  \int\displaylimits_{\substack{([N^{1+\theta},\infty)\times \rfour)  \mathsmaller{\bigcap} \KR} }|F_N| |w|^2 |\partial_t w| \dx \dt \\
&= \sum_{n=0}^{\lfloor N^\theta \rfloor} \sum_{\substack{j\in \zfour \\ \Kij \subseteq \KR}} \int_{\Kij} |F_N| |w|^2 |\partial_t w| \dx \dt +  \int\displaylimits_{\substack{([N^{1+\theta},\infty)\times \rfour)  \mathsmaller{\bigcap} \KR} }|F_N| |w|^2 |\partial_t w| \dx \dt~. 
\end{align*}
In the last line, we have used that \begin{equation*}
\KR = \bigcup_{ \substack{(n,j)\in \mathbb{N}_0\times \zfour\\ \Kij \subseteq \KR}} \Kij ~. 
\end{equation*}
We first control the contributions on the time intervals \( [\tau_n,\tau_n+N] \). To this end, 
we define \( w^{(N,n,j)} \) as the \( \Kij \)-locally forced solution with data 
\begin{equation*}
w^{(N,n,j)}(\tau_n)= w(\tau_n) \qquad \text{and} \qquad \partial_t w^{(N,n,j)}(\tau_n) = \partial_t w(\tau_n)~. 
\end{equation*}
Using finite speed of propagation, \( w \) and \( w^{(N,n,j)} \) coincide on \( \Kij \). 
Applying Proposition \ref{apriori:prop_single_scale} and the induction hypothesis to \( w^{(N,n,j)} \), it follows that 

\begin{align*}
&\int_{\Kij} |F_N| |w|^2 |\partial_t w| \dx \dt\displaybreak[3] \\
&= \int_{\Kij} |F_N| |w^{(N,n,j)}|^2 |\partial_t w^{(N,n,j)}| \dx \dt \displaybreak[3]\\ 
&\lesssim \eta N^{\frac{3}{4}-s+8\delta} \Eij[w^{(N,n,j)}]^{\frac{1}{2}} \left( \Eij[w^{(N,n,j)}]^{\frac{1}{2}}  + \Fij[w^{(N,n,j)}] \right)^{\frac{1}{2}}\displaybreak[3] \\
&\lesssim \eta N^{\frac{3}{4}-s+8\delta} \Big(2 E_{|x-y_j|\leq 16N}[w^{(N,n,j)}](\tau_n) + C_1 \| F \|_{L_t^3 L_x^6(\Kij)}^6 \Big)^\frac{1}{2} \\
&~~~ \cdot  N^{50\delta} \Big( C_1  E_{|x-y_j|\leq 16N}[w^{(N,n,j)}](\tau_n)  + C_1 \| F \|_{L_t^3L_x^6(\Kij)}^6 \Big)^{\frac{1}{2}} \\
&\lesssim \eta N^{\frac{3}{4}-s+ 60 \delta} C_1  \left(   E_{|x-y_j|\leq 16N}[w^{(N,n,j)}](\tau_n)  +  \| F \|_{L_t^3L_x^6(\Kij)}^6 \right) \\
&\lesssim \eta N^{\frac{3}{4}-s+ 60 \delta} C_1  \left(   E_{|x-y_j|\leq 16N}[w](\tau_n)  +  \| F \|_{L_t^3L_x^6(\Kij)}^6 \right) ~. 
\end{align*}
As a consequence, we obtain that 
\begin{align}
&\sum_{n=0}^{\lfloor N^\theta \rfloor} \sum_{\substack{j\in \zfour \\ \Kij \subseteq \KR}} \int_{\Kij} |F_N| |w|^2 |\partial_t w| \dx \dt  \notag \\
&\lesssim \eta N^{\frac{3}{4}-s+ 60\delta} C_1\sum_{n=0}^{\lfloor N^\theta \rfloor} \sum_{\substack{j\in \zfour \\ \Kij \subseteq \KR}}   \left(  E_{|x-y_j|\leq 16N}[w](\tau_n)  + \| F \|_{L_t^3L_x^6(\Kij)}^6 \right) \notag \\
&\lesssim \eta N^{\frac{3}{4}-s + 60 \delta} C_1 \sum_{n=\max(0,t_0/N)}^{\lfloor N^\theta \rfloor} \left(  E_{|x-x_0|\leq 16R- |t-t_0|}[w](\tau_n) +  \| F \|_{L_t^3L_x^6(([\tau_n,\tau_n+N]\times \rfour)~\mathsmaller{\bigcap}~ \KR)}^6 \right) \notag \\
&\lesssim \eta N^{\frac{3}{4}+ \theta -s + 60 \delta}  C_1 \ERtilde + \eta N^{\frac{3}{4}-s+ 60 \delta} C_1 \| F \|_{L_t^3 L_x^6(\KR)}^6~. \label{apriori:eq_proof_short_time}
\end{align}
Using the long-time decay estimate, we can control the contribution on the interval \( [N^{1+\theta},\infty) \) by 
\begin{align}
& \quad \int\displaylimits_{\substack{([N^{1+\theta},\infty)\times \rfour)  \mathsmaller{\bigcap} \KR} }|F_N| |w|^2 |\partial_t w| \dx \dt  \notag \\
& \lesssim \| F_N \|_{L_t^1 L_x^\infty([N^{1+\theta},\infty)\times \rfour)} \| w \|_{L_t^\infty L_x^4(\KR)}^2 \| \partial_t w \|_{L_t^\infty L_x^2(\KR)}  \notag \\
&\lesssim \eta N^{1-\frac{\theta}{2}-s+\delta} \ERtilde~. \label{apriori:eq_proof_long_time}
\end{align}
Combining \eqref{apriori:eq_proof_short_time} and \eqref{apriori:eq_proof_long_time}, it follows that 
\begin{align}
&\sum_{N\leq R/2} \int_{\KR} |F_N| |w|^2 |\partial_t w| \dx \dt \notag \\
&\lesssim \eta C_1 \sum_{N\leq R/2} \left( N^{\frac{3}{4}+\theta-s+ 5 \delta} + N^{1-\frac{\theta}{2}- s } \right) \ERtilde  +  \eta  \Big(\sum_{N\leq R/2}  N^{\frac{3}{4}-s+ 5 \delta}   \Big) C_1  \| F \|_{L_t^3 L_x^6(\KR)}^6 \label{apriori:eq_proof_low_frequencies}\\
&\lesssim \eta C_1 \ERtilde + \eta C_1  \| F \|_{L_t^3 L_x^6(\KR)}^6 ~. \notag
\end{align}
Here, we have used that \( s > \max( 1 -\frac{\theta}{2}, \frac{3}{4}+ \theta) \), and that \( \delta = \delta(s,\theta) > 0 \) is sufficiently small. \\

\emph{Step 5: Finishing the proof.} 
At this point, we have proven all the necessary estimates on \( w \). It only remains to put them together, and use a ``kick back'' argument. 
From the energy increment \eqref{apriori:eq_proof_energy_increment}, the high frequency estimate \eqref{apriori:eq_proof_high_frequencies}, the low-frequency estimate \eqref{apriori:eq_proof_low_frequencies}, and Hölder's inequality, it follows that 
\begin{equation}\label{apriori:eq_proof_kickback_energy}
\ERtilde[w] \leq   E_{|x-x_0|\leq 16R}[w](t_0) + \eta C_1 \ERtilde[w]  + \eta R^{\frac{3}{4}-s+8\delta} \FluxRtilde[w] + \eta C_1  \| F \|_{L_t^3 L_x^6(\KR)}^6  
\end{equation}
Inserting the same bound for the energy increment into \eqref{prelim:lem_averaged_flux}, we also have that 
\begin{equation}\label{apriori:eq_proof_kickback_flux}
\FluxRtilde[w] \lesssim R^{40\delta} \left( \ERtilde[w] +  \eta R^{\frac{3}{4}-s+8\delta} \FluxRtilde[w] + \eta C_1  \| F \|_{L_t^3 L_x^6(\KR)}^6 \right)
\end{equation}
If the absolute constant \( \eta = \eta(C_1,\delta)> 0 \) is chosen sufficiently small, then \eqref{apriori:eq_proof_kickback_flux} implies that 
\begin{equation}\label{apriori:eq_proof_kickback_flux_two}
\FluxRtilde[w] \lesssim R^{40\delta} \left( \ERtilde[w] + \eta C_1 \| F \|_{L_t^3 L_x^6(\KR)}^6 \right)~.
\end{equation}
Inserting this into \eqref{apriori:eq_proof_kickback_energy}, we obtain \eqref{apriori:eq_induction_energy}. Finally, \eqref{apriori:eq_induction_energy} and \eqref{apriori:eq_proof_kickback_flux_two} imply \eqref{apriori:eq_induction_flux}. This completes the proof of the induction step.

\end{proof}

Using Proposition \ref{apriori:prop_induction_scales}, we now provide a short proof of the main result.
\begin{proof}[Proof of Theorem \ref{main_theorem}]
Assume that the statements in Proposition \ref{wp:prop_smallness} hold for \( \omega \in \Omega \). From Lemma \ref{prelim:lem_lwp}, it follows that there exists a local solution to \eqref{prelim:eq_forced_nlw}. From Proposition \ref{apriori:prop_induction_scales}, it follows for all \( R \geq 1 \) that 
\begin{equation*}
\sup_{t\in[0,R]} \int_{|x|\leq 2R-t} T^{00}[v](t,x) \dx \leq 2E[v_0,v_1]+ C_1 ~. 
\end{equation*}
By letting \( R\rightarrow \infty \), we obtain the a-priori energy bound
\begin{equation*}
\sup_{t\in [0,\infty)} E[v](t) \leq 2E[v_0,v_1] + C_1~. 
\end{equation*}
From Proposition \ref{prelim:prop_reduction}, this implies the global space-time bound \( \| v\|_{L_t^3 L_x^6([0,\infty) \times \rfour)} < \infty \) and the existence of scattering states \((v_0^+ ,v_1^+ ) \in \dot{H}^1(\rfour)\times L^2(\rfour) \). Since \( u = F + v \), we obtain the global space-time bound \( \| u \|_{L_t^3 L_x^6([0,\infty)\times \rfour)} < \infty \) and the scattering states \( (u_0^+ , u_1^+ ) = (v_0^+- f^\omega_{0,<N_{\hi}}, v_1^+ - f^\omega_{1,<N_{\hi}} ) \). This completes the proof for positive times. By time-reflection symmetry, we obtain the same result for negative times.
\end{proof}

\renewcommand*{\bibfont}{\small}
\bibliography{Library_Wiki}
\bibliographystyle{hplain}
\end{document}